\documentclass[a4paper]{article}

\usepackage{graphicx,amssymb}
\usepackage{amsmath}
\usepackage{cases}
\usepackage{float}
\usepackage{listings}
\usepackage{times}
\usepackage{tikz}
\usepackage{epic}
\usepackage{titlesec}
\usepackage{geometry}
\usepackage{amsthm}
\usepackage{bm}
\usepackage{stmaryrd}
\usepackage{subfigure}
\usepackage{indentfirst}
\usepackage{xcolor}
\usepackage{epstopdf}
\usepackage{enumerate}
\usepackage{enumitem}
\usepackage{hyperref}

\definecolor{red}{rgb}{1.00,0.00,0.00}
{\numberwithin{equation}{section}
\setlength{\parindent}{1em}

\newtheorem{theorem}{Theorem}[section]
\newtheorem{lemma}{Lemma}[section]

\newtheorem{example}{Example}[section]
\newtheorem{assumption}{Assumption}[section]

\newtheorem{corollary}{Corollary}[section]

\SetSymbolFont{stmry}{bold}{U}{stmry}{m}{n}

\newcommand{\normmm}[1]{{\left\vert\kern-0.25ex\left\vert
\kern-0.25ex\left\vert #1
    \right\vert\kern-0.25ex\right\vert\kern-0.25ex\right\vert}}
\geometry{left=3cm,right=3cm,top=4cm,bottom=2.5cm}

\begin{document}
\title{Staggered discontinuous Galerkin methods for \\the Helmholtz equations with large wave number}
% a priori and a posteriori error analysis}
\author{Lina Zhao\footnotemark[1]\qquad
\;Eun-Jae Park\footnotemark[1]\qquad
\;Eric T. Chung\footnotemark[3]}
\renewcommand{\thefootnote}{\fnsymbol{footnote}}
\footnotetext[1]{Department of Mathematics,The Chinese University of Hong Kong, Hong Kong Special Administrative Region. ({lzhao@math.cuhk.edu.hk})}
\footnotetext[1]{Department of Computational Science and Engineering, Yonsei University, Seoul 03722, Republic of Korea. ({ejpark@yonsei.ac.kr})}
\footnotetext[3]{Department of Mathematics,The Chinese University of Hong Kong, Hong Kong Special Administrative Region. ({tschung@math.cuhk.edu.hk})}

%\date{}
%
\maketitle

\textbf{Abstract:}
In this paper we investigate staggered discontinuous Galerkin method for the Helmholtz equation with large wave number on general quadrilateral and polygonal meshes. The method is highly flexible by allowing rough grids such as the trapezoidal grids and highly distorted grids, and at the same time, is numerical flux free. Furthermore, it allows hanging nodes, which can be simply treated as additional vertices. By exploiting a modified duality argument, the stability and convergence can be proved under the condition that $\kappa h$ is sufficiently small, where $\kappa$ is the wave number and $h$ is the mesh size. Error estimates for both the scalar and vector variables in $L^2$ norm are established. Several numerical experiments are tested to verify our theoretical results and to present the capability of our method for capturing singular solutions.


\pagestyle{myheadings} \thispagestyle{plain}
\markboth{L. Zhao} {SDG methods for the Helmholtz equations with large wave number}

%%%%%%%%%%%%%%%%%%%%%%%%%%%%%%%%%%%%%%%%%%%%%%%%%%%%%%%%%%%%%%%%%%%%%%%%%%%%%%%%%%%%%%%
%%%%%%%%%%%%%%%%%%%%%%%%%%%%%%%%%%%%%%%%%%%%%%%%%%%%%%%%%%%%%%%%%%%%%%%%%%%%%%%%%%%%%%%
\section{Introduction}
In this paper we develop a staggered discontinuous Galerkin (DG) method for solving the following Helmholtz problem
\begin{align}
-\Delta u-\kappa^2 u&=f\quad \mbox{in}\; \Omega,\label{eq:model1}\\
\nabla u\cdot\bm{n}+i\kappa u&=g \quad \mbox{on}\; \partial \Omega,\label{eq:model2}
\end{align}
where $\kappa>0$ is the wave number and $f\in L^2(\Omega)$ represents a harmonic source and $g\in L^2(\partial \Omega)$ is a given data function. Here, $\Omega$ is a polygonal or polyhedral domain in $\mathbb{R}^2$ and $i=\sqrt{-1}$ is the imaginary unit.

The Helmholtz problem has many practical applications in electrodynamics, especially in optics and in acoustic involving time harmonic wave propagation. The Helmholtz equation with large wave number is indefinite, which makes it difficult to design robust and accurate numerical methods for the Helmholtz problem. For a fixed polynomial order, the pollution effect can be reduced substantially but cannot be avoided in principle \cite{BabuskaSauter}. A rigorous analysis for one-dimensional Helmholtz problems and piecewise linear approximation has been given in \cite{Aziz88} and \cite{Douglas93}, respectively, under the condition that $\kappa^2 h$ is sufficiently small. However, the assumption on $\kappa^2 h$ is too restrictive and unsatisfactory from a practical point of view. In order to obtain more stable and accurate numerical approximations, a large amount of nonstandard methods have been proposed and analyzed. Among all the methods, we mention in particular the general DG method on regular meshes \cite{MelenkSauter13} with mesh condition $\kappa (\kappa h)^m\leq C_0$, the absolutely stable discontinuous Galerkin (DG) methods \cite{FengWu09,FengWu11,FengXing13} without any mesh constraint, continuous interior penalty finite element methods (CIP-FEM) \cite{ZhuWu13} with the mesh condition $\frac{\kappa h}{m}\leq C_0(\frac{m}{\kappa})^{\frac{1}{m+1}}$ and a new weak Galerkin (WG) finite element method \cite{MuWangYe15} with the mesh condition $\kappa^{7/2}h^2\leq C_0$ or $(\kappa h)^2+\kappa(\kappa h)^{m+1}\leq C_0$, where $m$ is the polynomial order. In addition to the above mentioned approaches, many other numerical methods have also been developed, such as the partition of unity finite element methods \cite{BabuskaMelenk97,MelenkBabuska96}, the least squares finite element methods \cite{Chang90,Harari91,Harari92,MonkWang99,ChenQiu16}, the generalized finite element methods \cite{BabuskaGFEM95,BabuskaGFEM98}, the hybridized discontinuous Galerkin methods \cite{Griesmaier11}, the interior penalty discontinuous Galerkin methods \cite{Ainsworth06} and the Petrov Galerkin methods \cite{Demkowicz12,Gallistl15}.

Staggered DG method is initially developed for wave propagation problems \cite{ChungEngquistwave,ChungEngquist} on triangular meshes. Since then, it have been successfully applied to a large amount of partial differential equations arising from practical applications, see, e.g., \cite{ChungQiu,ChungCiarlet13Yu,ChungKimWidlund13,KimChungLee,
ChungCockburn14,KimChungLam16,LeeKim16,ChungParkLina,LinaParkStoke-tri,LinaParkconvection} and the references therein. Recently, staggered DG methods have been designed on general quadrilateral and polygonal meshes to solve Darcy law and the Stokes equations \cite{LinaPark,LinaParkShin}. The key features of staggered DG methods can be summarized as follows: First, it can preserve the physical properties, such as the local and global mass conservation, and achieve the superconvergent estimates. Second, it can be flexibly applied to rough grids such as the highly distorted grids and polygonal grids, and at the same time hanging nodes can be simply incorporated into the method. Third, thanks to the staggered continuity property, no numerical flux is needed in the construction of the method. All these distinctive properties make staggered DG method competitive in real applications.

The goal of this paper is to extend staggered DG method on general quadrilateral and polygonal meshes to the Helmholtz problem with large wave number. To the best of our knowledge, very few results are available for the Helmholtz problem on general meshes in the existing literature (cf. \cite{MuWangYe15}). The key idea for staggered DG method is to divide the initial partion (quadrilateral or polygonal meshes) into the union of triangles, then the primal mesh, the dual mesh and the primal simplexes can be constructed. Next, two sets of basis functions for the scalar and vector variables, respectively with staggered continuity for the Helmholtz problem are defined. The primary difficulty of analyzing the Helmholtz problem lies in the strong indefiniteness of the problem which makes it hard to establish the stability for the numerical approximation. To analyze staggered DG method, we exploit a modified duality argument. In addition, the elliptic projections in the spirit of staggered DG method for Darcy problem are defined. The key idea employed here is to use the specially designed elliptic projections in the duality argument to bound the $L^2$ errors of the discrete solution by the elliptic projections under the condition that $\kappa h$ is small enough. It is worth mentioning that the superconvergence of the elliptic projections for the scalar variable is the crux to achieve the optimal convergence for both the scalar and vector variables in $L^2$ errors with explicit dependence on $\kappa$. In addition, our estimates are comparable to those given in \cite{ZhuWu13} for CIP-FEM. Note that our (undisplayed) analysis shows that if we apply traditional duality argument as proposed in \cite{ChenLuXu13}, then the stability estimates require that $\kappa^2h$ should be sufficiently small, but this mesh condition is too restrictive for large wave number $\kappa$. Thus, we turn to the modified duality argument and improve the mesh condition, namely, we only require $\kappa h$ is small enough. It is worth mentioning that this is the first result on staggered DG method for the Helmholtz problems.

The rest of the paper is organized as follows. In the next section, we present the construction of staggered DG method for the Helmholtz problem. Then in Section~\ref{sec:convergence} we analyze the convergence of staggered DG method, where a modified duality argument is exploited. Finally, some numerical experiments are carried out in Section~\ref{sec:numerical} to confirm the proposed theories.

\section{Staggered DG method}
The staggered DG method is based on a first order formulation of the Helmholtz problem \eqref{eq:model1}-\eqref{eq:model2}, which can be  written in mixed form as finding $(\bm{p},u)$ such that
\begin{equation}
\begin{split}
i\kappa \bm{p}&=-\nabla u\quad\; \mbox{in}\; \Omega,\\
\nabla \cdot \bm{p}+i\kappa u&=f\hspace{1cm}\mbox{in}\;\Omega,\\
-\bm{p}\cdot\bm{n}+ u&=g\hspace{1cm} \mbox{on}\;\partial \Omega.
\end{split}
\label{eq:first}
\end{equation}

Next, we will briefly introduce staggered DG method for the Helmholtz problem in mixed form (cf. \eqref{eq:first}) on general quadrilateral and polygonal meshes, and more details can be referred to \cite{LinaPark, LinaParkShin}. To begin, we construct three meshes: the primal mesh $\mathcal{T}_{u}$, the dual mesh $\mathcal{T}_{p}$, and the primal simplexes $\mathcal{T}_h$. For a polygonal domain $\Omega$, consider a general mesh $\mathcal{T}_{u}$ (of $\Omega$) that consists of nonempty connected close disjoint subsets of $\Omega$ (see Figure~\ref{grid}):
\begin{align*}
\bar{\Omega}=\bigcup_{T\in \mathcal{T}_{u}}T.
\end{align*}
We also let $\mathcal{F}_{u}$ be the set of all primal edges in this partition and $\mathcal{F}_{u}^{0}$ be the
subset of all interior edges, that is, the set of edges in $\mathcal{F}_{u}$ that do not lie on $\partial\Omega$. We construct the primal submeshes $\mathcal{T}_h$ as a triangular subgrid of the primal grid: for an element $T\in \mathcal{T}_{u}$, elements of $\mathcal{T}_h$ are obtained by connecting the interior point $\nu$ to all vertices of $\mathcal{T}_{u}$ (see Figure~\ref{grid}):
\begin{align*}
\bar{\Omega}=\bigcup_{\tau\in \mathcal{T}_h}\bar{\tau}.
\end{align*}
We rename the union of these triangles by $S(\nu)$.
Moreover, we will use $\mathcal{F}_{p}$ to denote the set of all the dual edges generated by this subdivision process. For each triangle
$\tau\in \mathcal{T}_h$, we let $h_\tau$ be the diameter of
$\tau$ and $h=\max\{h_\tau, \tau\in \mathcal{T}_h\}$.
%Furthermore,
%$\mathcal{T}_h$ is assumed to satisfy the local quasi-uniform assumption in the sense that for any pair of elements $\tau$ and
%$\tau'$ in $\mathcal{T}_h$ which share an edge, there exists a
%constant $\kappa$ independent of $h_\tau$ and $h_{\tau'}$ such
%that $\kappa^{-1} \le h_\tau/h_{\tau'} \le \kappa$.
In addition, we define $\mathcal{F}:=\mathcal{F}_{u}\cup \mathcal{F}_{p}$ and $\mathcal{F}^{0}:=\mathcal{F}_{u}^{0}\cup \mathcal{F}_{p}$.
The construction for general meshes is illustrated in Figure~\ref{grid},
where the black solid lines are edges in $\mathcal{F}_{u}$
and the red dotted lines are edges in $\mathcal{F}_{p}$.

Finally, we construct the dual mesh. For each interior edge $e\in \mathcal{F}_{u}^0$, we use $D(e)$ to denote the dual mesh, which is the union of the two triangles in $\mathcal{T}_h$ sharing the edge $e$,
and for each boundary edge $e\in\mathcal{F}_{u}\backslash\mathcal{F}_{u}^0$, we use $D(e)$ to denote the triangle in $\mathcal{T}_h$ having the edge $e$,
see Figure~\ref{grid}.
We write $\mathcal{T}_{p}$ as the union of all $D(e)$.

For each edge $e$, we define
a unit normal vector $n_{e}$ as follows: If $e\in \mathcal{F}\setminus \mathcal{F}^{0}$, then
$\bm{n}_{e}$ is the unit normal vector of $e$ pointing towards the outside of $\Omega$. If $e\in \mathcal{F}^{0}$, an
interior edge, we then fix $\bm{n}_{e}$ as one of the two possible unit normal vectors on $e$.
When there is no ambiguity,
we use $\bm{n}$ instead of $\bm{n}_{e}$ to simplify the notation.
\begin{figure}
\centering
\includegraphics[width=12cm]{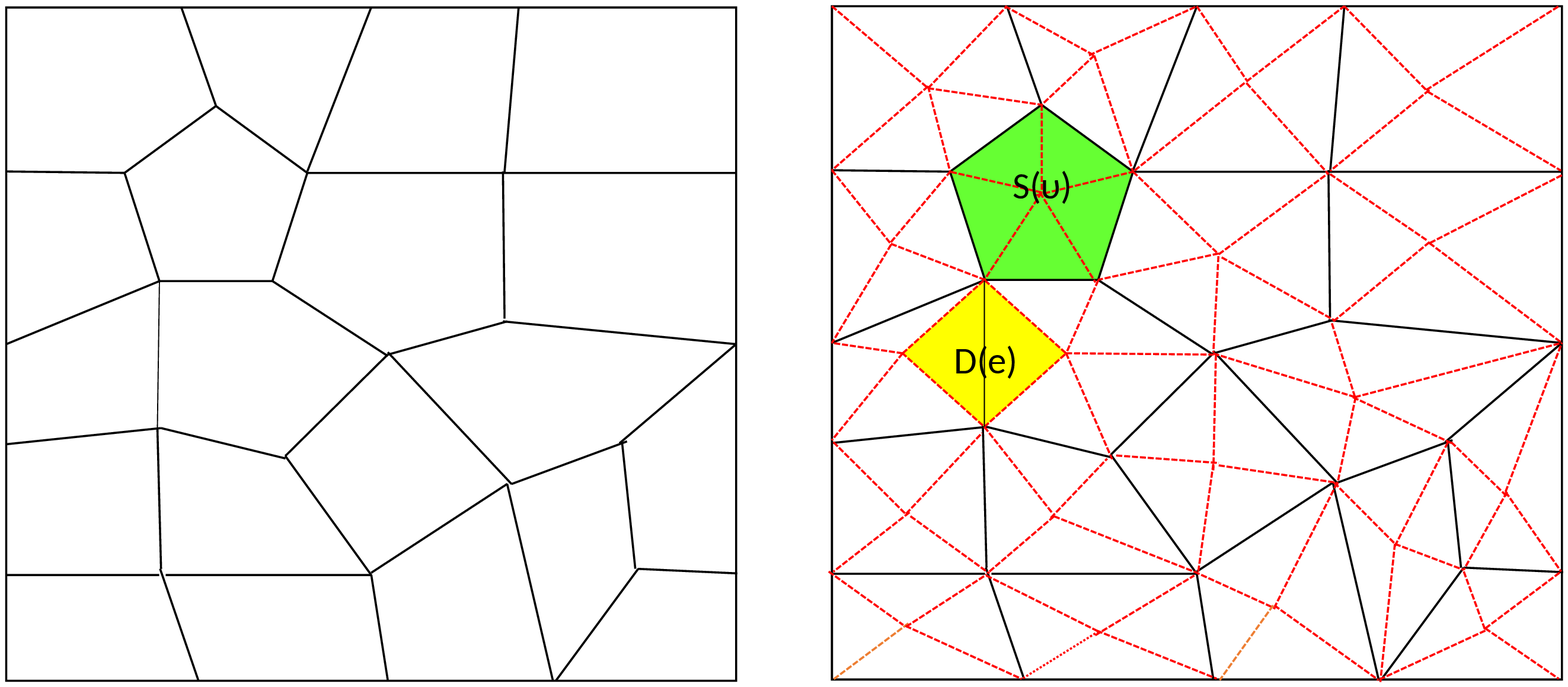}
\setlength{\abovecaptionskip}{-0.5cm}
\caption{Schematic of the primal mesh $S(\nu)$, the dual mesh $D(e)$ and the primal simplexes.}
\label{grid}
\end{figure}

%\begin{figure}[H]
%\centering
%\scalebox{0.3}{
%\includegraphics[width=20cm]{figs/primal.eps}
%}
%\scalebox{0.3}{
%\includegraphics[width=20cm]{figs/mesh1.eps}
%}
%\caption{Schematic of the primal mesh $S(\nu)$, the dual mesh $D(e)$ and the primal simplexes.}
%\label{grid}
%\end{figure}

Let $D\subset \mathbb{R}^2$, we adopt the standard notations for the Sobolev spaces $H^s(D)$ and their associated norms $\|\cdot\|_{s,D}$, and semi-norms $|\cdot|_{s,D}$ for $s\geq 0$. In particular, $(\cdot,\cdot)_D$ and $\langle \cdot,\cdot\rangle_\Sigma$ for $\Sigma\subset D$ denote the $L^2$ inner product on complex valued $L^2(D)$ and $L^2(\Sigma)$, respectively. If $D=\Omega$, the subscript $\Omega$ will be dropped unless otherwise mentioned. In the sequel, we use $C$ to denote a generic positive constant which may have different values at different occurrences.

The following mesh regularity assumptions are also needed throughout the paper (cf. \cite{Beir13,Cangiani16}).
\begin{assumption}\label{assum:regularity}
We assume there exists a constant $\rho>0$ such that
\begin{enumerate}[itemindent=1em,,label=(\arabic*)]
  \item For every element $S(\nu)\in \mathcal{T}_{u}$ and every edge $e\in \partial S(\nu)$, it satisfies $h_e\geq \rho h_{S(\nu)}$, where $h_e$ denotes the length of edge $e$ and $h_{S(\nu)}$ denotes the diameter of $S(\nu)$.
  \item Every element $S(\nu)$ in $\mathcal{T}_{u}$ is star-shaped with respect to a ball of radius $\geq \rho h_{S(\nu)}$.
\end{enumerate}
\end{assumption}
We remark that the above assumptions ensure that the triangulation
$\mathcal{T}_h$ is shape regular.

Let $m\geq 0$ be the order of approximation. For every $\tau
\in \mathcal{T}_{h}$ and $e\in\mathcal{F}$,
we define $P^{m}(\tau)$ and $P^{m}(e)$ as the spaces of polynomials of degree less than or equal to $m$ on $\tau$ and $e$, respectively. For $w$ and $\bm{v}$ belonging to the broken Sobolev space,
the jump $[v]$ and the jump $[\bm{v}\cdot\bm{n}]$ are defined respectively as
\begin{equation*}
[w]=w_{1}-w_{2}, \quad [\bm{v}\cdot\bm{n}]=\bm{v}_{1}\cdot\bm{n}-\bm{v}_{2}\cdot\bm{n},
\end{equation*}
where $v_{i}=v\mid_{\tau_{i}}$, $\bm{v}_{i}=\bm{v}\mid_{\tau_{i}}$ and $\tau_{1}$, $\tau_{2}$ are the two triangles in $\mathcal{T}_h$ having the edge $e\in \mathcal{F}$.
In the above definitions, we assume $\bm{n}$ is pointing from $\tau_1$ to $\tau_2$.

Next, we will introduce some finite dimensional spaces.
First, we define the following locally $H^{1}(\Omega)$ conforming
space $S_h$:
\begin{equation*}
S_{h}:=\{w : w\mid_{\tau}\in P^{m}(\tau)\; \forall \tau \in \mathcal{T}_{h}; [w]\mid_e=0\;\forall e\in
\mathcal{F}_{u}^0\}.
\end{equation*}
Notice that, if $w\in S_h$, then $w\mid_{D(e)}\in H^1(D(e))$ for each edge $e\in \mathcal{F}_{u}$.
We next define the following locally $H(\mbox{div};\Omega)-$conforming
SDG space $\bm{V}_h$:
\begin{equation*}
\bm{V}_{h}=\{\bm{v}: \bm{v}\mid_{\tau} \in P^{m}(\tau)^{2}\;
\forall \tau \in \mathcal{T}_{h};[\bm{v}\cdot\bm{n}]\mid_e=0\;
\forall e\in \mathcal{F}_{p}\}.
\end{equation*}
Note that if $\bm{v}\in \bm{V}_h$, then $\bm{v}\mid_{S(\nu)}\in H(\textnormal{div};S(\nu))$ for each $S(\nu)\in \mathcal{T}_{u}$.

Following \cite{LinaPark, LinaParkShin}, we can obtain the staggered DG formulation for \eqref{eq:first}: find $(u_h, \bm{p}_h)\in S_h\times \bm{V}_h$ such that
\begin{align}
(i\kappa \bm{p}_h, \overline{\bm{q}}_h)&=b_h^*(u_h, \overline{\bm{q}}_h)\quad \forall \bm{q}_h\in \bm{V}_h, \label{eq:SDG1}\\
b_h(\bm{p}_h, \overline{v}_h)+i\kappa (u_h,\overline{v}_h)+ \langle u_h,\overline{v}_h\rangle_{\partial \Omega}&=(f, \overline{v}_h)+\langle g,\overline{v}_h\rangle_{\partial \Omega}\quad \forall v_h\in S_h,\label{eq:SDG2}
\end{align}
where the bilinear forms are defined as
\begin{align*}
b_h(\bm{q}_h, v_h)&=-(\bm{q}_h, \nabla v_h)+\sum_{e\in \mathcal{F}_p}\langle \bm{q}\cdot\bm{n},[v_h]\rangle_e\quad \forall (\bm{q}_h, v_h)\in \bm{V}_h\times S_h, \\
b_h^*(v_h, \bm{q}_h)&=(\nabla \cdot \bm{q}_h, v_h)-\sum_{e\in \mathcal{F}_u}\langle[\bm{q}_h\cdot\bm{n}],v_h\rangle_e\quad \forall (\bm{q}_h, v_h)\in \bm{V}_h\times S_h.
\end{align*}

The following discrete adjoint property can be verified easily by integration by parts:
\begin{align}
b_h(\bm{q}_h, \overline{v}_h)=b_h^*(\overline{v}_h, \bm{q}_h)\quad \forall (v_h,\bm{q}_h)\in S_h\times \bm{V}_h.\label{eq:adjoint}
\end{align}

%Recall that $\bm{p}_h, u_h$ satisfies the following system
%\begin{align}
%i\kappa (\bm{p}_h,\overline{q}_h)&=b_h^*(u_h,\overline{q}_h),\\
%b_h(\bm{p}_h,\overline{v}_h)+i\kappa (u_h,\overline{v}_h)+\langle u_h, \overline{v}_h\rangle_{\partial \Omega}&=(f,\overline{v}_h)+\langle g, \overline{v}_h\rangle_{\partial \Omega}.
%\end{align}
Let $A_h(\bm{p}_h,u_h;\bm{q}_h,v_h)=-i\kappa (\overline{\bm{p}}_h,\bm{q}_h)-b_h^*(\overline{u}_h,\bm{q}_h)+b_h(\bm{p}_h,\overline{v}_h)$. Then we have from integration by parts
\begin{align}
A_h(\bm{p}-\bm{p}_h,u-u_h;\bm{q}_h,v_h)+i\kappa (u-u_h,\overline{v}_h)+\langle u-u_h,\overline{v}_h\rangle_{\partial \Omega}=0\quad \forall (\bm{q}_h,v_h)\in \bm{V}_h\times S_h.\label{eq:defA}
\end{align}

%\begin{lemma}(trace inequality).
%
%\begin{align*}
%\|v\|_{0,\tau}\leq
%\end{align*}
%
%
%\end{lemma}

\section{Convergence analysis}\label{sec:convergence}
In this section, we aim to derive the convergence estimates and stability of staggered DG method for the Helmholtz problem, our (undisplayed) analysis shows that standard duality argument requires $\kappa ^2h$ be sufficiently small to achieve the stability. To overcome this issue, a modified duality argument is exploited, where the elliptic projections are the key tools. From which we can get the stability and convergence estimates provided $\kappa h$ is small enough.
%Proceeding analogously to \cite{ChenLuXu13}, we can obtain the next lemma.
%\begin{lemma}\label{lemma:regularity}
%
%\begin{align*}
%\kappa \|\psi\|_{0}+\kappa^{-1}\|\psi\|_{2,\Omega}+\kappa^{-1}\|\psi\|_{1,\Omega}
%+\|\psi\|_{0,\partial \Omega}+\kappa^{-1}\|\bm{\Phi}\|_1\leq C \|u_h\|_0.
%\end{align*}
%\end{lemma}

To begin, we define the following projection operators, which will be useful for the subsequent analysis.
Let $I_h: H^1(\Omega)\rightarrow S_h$ be defined by
\begin{align*}
\langle I_h v-v,\overline{\phi}\rangle_e &=0 \quad \forall \phi\in P^m(e), e\in \mathcal{F}_u,\\
(I_hv-v,\overline{\phi})_\tau&=0\quad \forall \phi\in P^{m-1}(\tau),\tau\in \mathcal{T}_h.
\end{align*}
In addition, $J_h: H^1(\Omega)\rightarrow \bm{V}_h$ is defined by
\begin{align*}
\langle(J_h\bm{q}-\bm{q})\cdot\bm{n},\overline{v}\rangle_e&=0\quad \forall v\in P^{m}(e),\forall e\in \mathcal{F}_p,\\
(J_h\bm{q}-\bm{q}, \overline{\bm{\phi}})_\tau&=0\quad \forall \bm{\phi}\in P^{m-1}(\tau)^2, \forall \tau\in \mathcal{T}_h.
\end{align*}
By the definitions of $I_h$ and $J_h$, we can get
\begin{align}
b_h(\bm{p}-J_h\bm{p},\overline{v})&=0\quad \forall v\in S_h,\\
b_h^*(u-I_hu, \overline{\bm{q}})&=0\quad \forall \bm{q}\in\bm{V}_h.\label{eq:uIhu}
\end{align}
In addition, we can also derive from the definition of the bilinear form $A_h(\cdot,\cdot;\cdot,\cdot)$ (cf. \eqref{eq:defA})
\begin{equation}
\begin{split}
&A_h(J_h\bm{p}-\bm{p}_h,I_hu-u_h;\bm{q}_h,v_h)+i\kappa (I_hu-u_h,\overline{v}_h)+\langle I_hu-u_h,\overline{v}_h\rangle_{\partial \Omega} \\ &=A_h(J_h\bm{p}-\bm{p},I_hu-u;\bm{q}_h,v_h)+i\kappa (I_hu-u,\overline{v}_h)+\langle I_hu-u,\overline{v}_h\rangle_{\partial \Omega}\\
&=-i\kappa (\overline{J_h\bm{p}-\bm{p}},\bm{q}_h)+i\kappa (I_hu-u,\overline{v}_h)\quad \forall (\bm{q}_h,v_h)\in \bm{V}_h\times S_h.
\end{split}
\label{eq:errorJhp}
\end{equation}

The next lemma is found to be useful for the subsequent analysis, one can refer to \cite{MelenkSauter11} for proof.
\begin{lemma}\label{lemma:ubound}
The solution $u$ to the problem \eqref{eq:model1} and \eqref{eq:model2} can be written as $u=u_\epsilon+u_{\mathcal{A}}$, and satisfies
\begin{align*}
|u_\epsilon|_j&\leq C\kappa^{j-2} M_{f,g}\quad j=0,1,2,\\
|u_{\mathcal{A}}|_j&\leq C \kappa^{j-1}M_{f,g}\quad \forall j\in \mathbb{N}_0,
\end{align*}
where $M_{f,g}=\|f\|_0+\|g\|_{1/2,\partial \Omega}$.

\end{lemma}

The approximation estimates for $I_h$ and $J_h$ with high order convergence are given as follows (cf. \cite{Ciarlet78,ChungEngquist}).
\begin{lemma}\label{lemma:appro1}
Assume that $u\in H^{m+1}(\Omega)$, then we have
\begin{align*}
\|u-I_hu\|_0+h\|\nabla (u-I_hu)\|_0&\leq C h^m\|u\|_m,\\
\|\bm{p}-J_h\bm{p}\|_0&\leq C h^m\|\bm{p}\|_m.
\end{align*}

\end{lemma}

For our convergence analysis, we also need the following approximation estimates, which only require $H^2$ regularity. It uses the decomposition given in Lemma~\ref{lemma:ubound} and is more subtle than the estimates given in Lemma~\ref{lemma:appro1}. The proof is analogous to that of \cite{MelenkSauter10}, thus we omit the proof for simplicity.
\begin{lemma}\label{lemma:approximation}(approximation properties)
Let $(\bm{p},u)$ be the solution to \eqref{eq:first}, then
\begin{align*}
\kappa \|\bm{p}-J_h\bm{p}\|_0&\leq C (h+(\kappa h)^m)M_{f,g},\\
\|u-I_hu\|_0&\leq C (h^2+h(\kappa h)^m)M_{f,g},\\
\|\nabla (u-I_hu)\|_0&\leq C (h+(\kappa h)^m)M_{f,g}.
\end{align*}

\end{lemma}

%\begin{proof}
%Since $i\kappa \bm{p}=-\nabla u$, we can decompose $\bm{p}$ into the sum of an ellitpic part and an analytic part, i.e., $\bm{p}=\bm{p}_{\epsilon}+\bm{p}_A$
%
%\end{proof}

The rest of this section is devoted to the error analysis for the Helmholtz problem, to this end we introduce the elliptic projections motivated by those proposed in \cite{ZhuWu13,ZhuDu15}. For any $u,\bm{p}$, we define its elliptic projections $u_h^+, \bm{p}_h^+$ as the staggered DG approximations to the first order system
\begin{align*}
i\kappa \bm{p}&=-\nabla u\quad \mbox{in}\;\Omega,\\
\nabla \cdot\bm{p}&=F\hspace{0.9cm} \mbox{in}\;\Omega,\\
-\bm{p}\cdot\bm{n}+u& = G\hspace{0.9cm} \mbox{on}\;\partial\Omega.
\end{align*}
Then $(\bm{p}_h^+,u_h^+)\in \bm{V}_h\times S_h$ is the numerical solution to the following discrete formulation
\begin{align*}
i\kappa (\bm{p}_h^+,\overline{\bm{q}}_h)&=b_h^*(u_h^+,\overline{\bm{q}}_h)\quad \forall \bm{q}_h\in \bm{V}_h,\\
b_h(\bm{p}_h^+,\overline{v}_h)+\langle u_h^+, \overline{v}_h\rangle_{\partial \Omega}&=(F,\overline{v}_h)+\langle G, \overline{v}_h\rangle_{\partial \Omega}\quad \forall v_h\in S_h.
\end{align*}
Let $B_h(\bm{p}_h,u_h;\bm{q}_h,v_h)=i\kappa (\bm{p}_h,\overline{\bm{q}}_h)-b_h^*(u_h,\overline{\bm{q}}_h)+b_h(\bm{p}_h,\overline{v}_h)$, then it follows
\begin{align}
B_h(\bm{p}_h^+,u_h^+;\bm{q}_h,v_h)+\langle u_h^+,\overline{v}_h\rangle_{\partial \Omega}=B_h(\bm{p},u;\bm{q}_h,v_h)+\langle u,\overline{v}_h\rangle_{\partial \Omega}\quad \forall (\bm{q}_h,v_h)\in \bm{V}_h\times S_h,\label{eq:defproj1}
\end{align}
which immediately yields
\begin{align}
i\kappa (\bm{p}-\bm{p}_h^+,\overline{\bm{q}}_h)&=b_h^*(u-u_h^+,\overline{\bm{q}}_h)\quad \forall \bm{q}_h\in \bm{V}_h,\\
b_h(\bm{p}-\bm{p}_h^+, \overline{v}_h)+\langle I_hu-u_h^+, \overline{v}_h\rangle_{\partial \Omega}&=0\quad \forall v_h\in S_h.\label{eq:errorph+}
\end{align}
%Consequently
%\begin{align}
%B_h(\bm{p}-\bm{p}_h^+,u-u_h^+;\bm{q}_h,v_h)+\langle u-u_h^+,\overline{v}_h\rangle_{\partial \Omega}=0\quad \forall (\bm{q}_h,v_h)\in \bm{V}_h\times S_h.\label{eq:projection1}
%\end{align}

In addition, we define $\bm{p}_h^-,u_h^-$ by
\begin{align}
B_h(\bm{q}_h,v_h;\bm{p}_h^-,u_h^-)+\langle v_h,\overline{u_h^-}\rangle_{\partial \Omega}=B_h(\bm{q}_h,v_h;\bm{p},u)+\langle v_h,\overline{u}\rangle_{\partial \Omega}\quad \forall (\bm{q}_h,v_h)\in \bm{V}_h\times S_h.\label{eq:defproj2}
\end{align}
The following holds
\begin{align}
B_h(\bm{q}_h,v_h;\bm{p}-\bm{p}_h^-,u-u_h^-)+\langle v_h,\overline{u-u_h^-} \rangle_{\partial \Omega}=0\quad \forall (\bm{q}_h,v_h)\in \bm{V}_h\times S_h.\label{eq:projection2}
\end{align}

Indeed, we have
\begin{align*}
\overline{B_h(\bm{q}_h,v_h;\bm{p},u)}=-B_h(\bm{p},u;\bm{q}_h,v_h)\quad \forall (\bm{q}_h,v_h)\in \bm{V}_h\times S_h.
\end{align*}
To be specific,we can obtain from \eqref{eq:defproj2}
\begin{align*}
-B_h(\bm{p}_h^-,u_h^-;\bm{q}_h,v_h)+\langle u_h^-,\overline{v}_h\rangle_{\partial \Omega}=-B_h(\bm{p},u;\bm{q}_h,v_h)+\langle u,\overline{v}_h\rangle_{\partial \Omega}\quad \forall (\bm{q}_h,v_h)\in \bm{V}_h\times S_h.
\end{align*}
In other words, $(\bm{p}_h^-,u_h^-)\in \bm{V}_h\times S_h$ is the staggered DG approximation to the following first order system
\begin{align*}
i\kappa \bm{p}&=-\nabla u\quad \mbox{in}\;\Omega,\\
\nabla \cdot\bm{p}&=F\hspace{0.9cm}\mbox{in}\;\Omega,\\
\bm{p}\cdot\bm{n}+u& = G\hspace{0.9cm}\mbox{on}\;\partial \Omega.
\end{align*}

%\begin{lemma}
%
%\begin{align*}
%\|I_hu-u_h^{\pm}\|_Z\leq C \|J_h\bm{p}-\bm{p}_h^{\pm}\|_0.
%\end{align*}
%
%\end{lemma}
%
%
%\begin{proof}
%\Red{Check whether the usage of inf-sup condition is correct or not}
%\begin{align*}
%\|I_hu-u_h^+\|_Z\leq C\sup_{\bm{q}_h\in \bm{V}_h} |\frac{b_h(\bm{q}_h,I_hu-u_h^+)}{\|\bm{q}_h\|_0}|=C \sup_{\bm{q}_h\in \bm{V}_h} |\frac{(\bm{p}-\bm{p}_h^+,\bm{q}_h)}{\|\bm{q}_h\|_0}|\leq C \|\bm{p}-\bm{p}_h^+\|_0.
%\end{align*}
%
%
%\end{proof}
On the other hand we have the following identity by integration by parts
\begin{align*}
A_h(\bm{p}-\bm{p}_h^+,u-u_h^+;\bm{q}_h,v_h)+\langle u-u_h^+,\overline{v}_h\rangle_{\partial \Omega}=0\quad \forall (\bm{q}_h,v_h)\in \bm{V}_h\times S_h.
\end{align*}
In addition, if $(\bm{p},u)$ is the solution of \eqref{eq:first} then
\begin{align}
B_h(\bm{p}-\bm{p}_h,u-u_h;\bm{q}_h,v_h)+i\kappa (u-u_h,\overline{v}_h)+\langle u-u_h, \overline{v}_h\rangle_{\partial \Omega}=0\quad \forall (\bm{q}_h,v_h)\in \bm{V}_h\times S_h\label{eq:errorB}.
\end{align}

Now we are ready to prove the next lemma.
\begin{lemma}\label{lemma:uh+}
Assume $u$ is any function in $H^2(\Omega)$ and $\bm{p}$ is any function in $H^1(\Omega)^2$. Then it holds
\begin{align*}
\|I_hu-u_h^{\pm}\|_0&\leq C h \|\bm{p}-J_h\bm{p}\|_0,\\
\|J_h\bm{p}-\bm{p}_h^{\pm}\|_0&\leq C \|\bm{p}-J_h\bm{p}\|_0.
\end{align*}

\end{lemma}

\begin{proof}
The proof for the elliptic projections $(u_h^+,\bm{p}_h^+)$ and $(u_h^-,\bm{p}_h^-)$ are similar, to simplify the presentation, we only give the proof for $(u_h^+,\bm{p}_h^+)$.

Proceeding analogously to \eqref{eq:errorJhp}, we can obtain
\begin{align*}
A_h(J_h\bm{p}-\bm{p}_h^+,I_hu-u_h^+;\bm{q}_h,v_h)+\langle I_hu-u_h^+, \overline{v}_h\rangle_{\partial \Omega}=-i\kappa (\overline{J_h\bm{p}-\bm{p}}, \bm{q}_h)\quad \forall (\bm{q}_h,v_h)\in \bm{V}_h\times S_h.
\end{align*}
Let $\bm{q}_h=J_h\bm{p}-\bm{p}_h^+, v_h=I_hu-u_h^+$, then
\begin{align}
A_h(J_h\bm{p}-\bm{p}_h^+, I_hu-u_h^+;J_h\bm{p}-\bm{p}_h^+, I_hu-u_h^+)+\|I_hu-u_h^+\|_{0,\partial \Omega}^2=-i\kappa (\overline{J_h\bm{p}-\bm{p}},J_h\bm{p}-\bm{p}_h^+).\label{eq:Ah}
\end{align}
On the other hand, we have from the definition of $A_h$ and the discrete adjoint property \eqref{eq:adjoint} that
\begin{align*}
A_h(J_h\bm{p}-\bm{p}_h^+, I_hu-u_h^+;J_h\bm{p}-\bm{p}_h^+,I_hu-u_h^+)=-i\kappa \|J_h\bm{p}-\bm{p}_h^+\|_0^2.
\end{align*}
Taking the imaginary part of \eqref{eq:Ah}, we can obtain
\begin{align*}
\kappa \|J_h\bm{p}-\bm{p}_h^+\|_0^2\leq C \kappa \|J_h\bm{p}-\bm{p}\|_0\|J_h\bm{p}-\bm{p}_h^+\|_0.
\end{align*}
Thus
\begin{align*}
\|J_h\bm{p}-\bm{p}_h^+\|_0\leq C \|J_h\bm{p}-\bm{p}\|_0.
\end{align*}

Next, we will estimate $\|I_hu-u_h^+\|_0$. Consider the auxiliary problem
\begin{equation}
\begin{split}
i\kappa \bm{\Phi}&=-\nabla \varphi\hspace{0.9cm} \mbox{in}\;\Omega,\\
\nabla \cdot \bm{\Phi}&=I_hu-u_h^+\quad\mbox{in}\;\Omega,\\
-\bm{\Phi}\cdot\bm{n}-\varphi&=0\hspace{1.6cm} \mbox{on}\;\partial \Omega,
\end{split}
\label{eq:dual-eta}
\end{equation}
which satisfies the following elliptic regularity estimate
\begin{align}
\kappa \|\bm{\Phi}\|_1+\|\varphi\|_2\leq C \|I_hu-u_h^+\|_0.\label{eq:ellip-regularity}
\end{align}

%\begin{align}
%i\kappa (\bm{p},\overline{\bm{q}}_h)&=b_h^*(u,\overline{\bm{q}}_h),\\
%b_h(\bm{p},\overline{v}_h)+\langle u, \overline{v}_h\rangle_{\partial \Omega}&=(F, \overline{v}_h)+\langle G, \overline{v}_h\rangle_{\partial \Omega}.
%\end{align}
%Thus
%\begin{align*}
%i\kappa (\bm{p}-\bm{p}_h^+,\overline{\bm{q}}_h)&=b_h^*(u-u_h^+,\overline{\bm{q}}_h),\label{eq:erroruh+}\\
%b_h(\bm{p}-\bm{p}_h^+, \overline{v}_h)+\langle I_hu-u_h^+, \overline{v}_h\rangle_{\partial \Omega}&=0.
%\end{align*}
An application of the Cauchy-Schwarz inequality, the discrete adjoint property \eqref{eq:adjoint}, \eqref{eq:uIhu}, \eqref{eq:errorph+}, \eqref{eq:dual-eta} and the elliptic regularity estimate \eqref{eq:ellip-regularity} lead to
\begin{align*}
\|I_hu-u_h^+\|_0^2&=(I_hu-u_h^+, \overline{I_hu-u_h^+})=(I_hu-u_h^+, \overline{\nabla \cdot \bm{\Phi}})-i\kappa (\overline{\bm{\Phi}},\bm{p}-\bm{p}_h^+)+(\overline{\nabla \varphi},\bm{p}-\bm{p}_h^+)\\
&=b_h(\overline{\bm{\Phi}},I_hu-u_h^+)-i\kappa(\overline{\bm{\Phi}},\bm{p}-\bm{p}_h^+)
-b_h^*(\overline{\varphi}, \bm{p}-\bm{p}_h^+)-\langle I_hu-u_h^+, \overline{\varphi}\rangle_{\partial \Omega}\\
&=-i\kappa(\bm{p}-\bm{p}_h^+, \overline{\bm{\Phi}-J_h\bm{\Phi}})+b_h^*(I_hu-u_h^+,\overline{\bm{\Phi}-J_h\bm{\Phi}})
-b_h(\bm{p}-J_h\bm{p},\overline{\varphi-I_h\varphi})\\
&\;-\langle I_hu-u_h^+, \overline{\varphi-I_h\varphi}\rangle_{\partial \Omega} \\
&\leq C \Big(\kappa \|\bm{p}-\bm{p}_h^+\|_0\|\bm{\Phi}-J_h\bm{\Phi}\|_0
+\|\bm{p}-J_h\bm{p}\|_{0}\|\nabla (\varphi-I_h\varphi)\|_0\Big)\\
&\leq C \Big(\kappa \|\bm{p}-J_h\bm{p}\|_0\|\bm{\Phi}-J_h\bm{\Phi}\|_0
+\|\bm{p}-J_h\bm{p}\|_{0}\|\nabla (\varphi-I_h\varphi)\|_0\Big)\\
&\leq C\Big(  h\|\bm{p}-J_h\bm{p}\|_0\|I_hu-u_h^+\|_0+h\|\bm{p}-J_h\bm{p}\|_0\|I_hu-u_h^+\|_0\Big)\\
&\leq C h\|\bm{p}-J_h\bm{p}\|_0\|I_hu-u_h^+\|_0.
\end{align*}
Therefore, the proof is complete.
%Therefore
%\begin{align}
%\|\eta\|_0\leq C h^2 \|\bm{p}\|_1.
%\end{align}

\end{proof}

Lemmas~\ref{lemma:approximation} and \ref{lemma:uh+} yield the next lemma.
\begin{lemma}\label{lemma:Ihu}
Let $(\bm{p},u)$ be the solution to the problem \eqref{eq:first}. Then we have
\begin{align*}
\|I_hu-u_h^+\|_0&\leq  C \frac{h}{\kappa}(h+(\kappa h)^m)M_{f,g},\\
\|J_h\bm{p}-\bm{p}_h^+\|_0&\leq C\frac{1}{\kappa}(h+(\kappa h)^m)M_{f,g}.
\end{align*}

\end{lemma}

The $L^2$ error estimates for both $\bm{p}_h$ and $u_h$ are stated in the next lemma.
\begin{lemma}\label{lemma:L2}
Let $(u,\bm{p})$ and $(u_h,\bm{p}_h)$ denote the solution of \eqref{eq:first} and \eqref{eq:SDG1}-\eqref{eq:SDG2}, respectively. If $\kappa h$ is sufficiently small, then the following estimates hold
\begin{align}
\|u-u_h\|_0&\leq C\Big(h+(\kappa h)^m\Big)\Big(\|\bm{p}-J_h\bm{p}\|_0+\|\nabla (u-I_hu)\|_0+\|u-I_hu\|_0\Big),\label{eq:L2u}\\
\|\bm{p}-\bm{p}_h\|_0&\leq C \Big(1+(\kappa h)^m\Big)\Big(\|\bm{p}-J_h\bm{p}\|_0+\|u-I_hu\|_0+\|\nabla (u-I_hu)\|_0\Big)\nonumber.
\end{align}

\end{lemma}

\begin{proof}

%Let $e_u=I_hu-u_h$, and $e_p= J_h\bm{p}-\bm{p}_h$
We first estimate the $L^2$ error of $u_h$ by introducing the dual problem and exploiting the elliptic projections of the solution to the original continuous problem and of the solution to the dual problem. We consider the following dual problem
\begin{equation}
\begin{split}
i\kappa \bm{\psi}&=-\nabla \varphi \hspace{0.5cm}\mbox{in}\;\Omega,\\
-\nabla \cdot\bm{\psi}-i\kappa \varphi&=u-u_h\quad\mbox{in}\;\Omega,\\
\bm{\psi}\cdot\bm{n}+\varphi&=0 \hspace{1.2cm}\mbox{on}\;\partial\Omega.
\end{split}
\label{eq:dualf}
\end{equation}
%Proceeding analogously to \cite{ChenLuXu13}, we can obtain the following estimate
%\begin{align}
%\kappa^{-1}\|\varphi\|_2+\|\varphi\|_1+\|\bm{\psi}\|_1\leq C \|I_hu-u_h\|_0.\label{eq:regularity}
%\end{align}
Let $(\bm{p}_h^+,u_h^+)\in \bm{V}_h\times S_h$ be the elliptic projections defined by \eqref{eq:defproj1} and let $(\bm{\psi}_h^-,\varphi_h^-)\in \bm{V}_h\times S_h$ be the elliptic projections defined by \eqref{eq:defproj2} by replacing $(\bm{p},u)$ by $(\bm{\psi},\varphi)$. It follows from Lemma~\ref{lemma:Ihu} by replacing $u$ by $\varphi$ that
\begin{equation}
\begin{split}
\|I_h\varphi-\varphi_h^-\|_0&\leq C \frac{h}{\kappa}(h+(\kappa h)^m)\|u-u_h\|_0,\\
\|J_h\bm{\psi}-\bm{\psi}_h^-\|_0&\leq C \frac{1}{\kappa}(h+(\kappa h)^m)\|u-u_h\|_0.
\end{split}
\label{eq:Ihvarphi}
\end{equation}
In addition, Lemma~\ref{lemma:approximation} yields
\begin{equation}
\begin{split}
\|\varphi-I_h\varphi\|_0&\leq C h (h+(\kappa h)^m)\|u-u_h\|_0,\\
\kappa \|\bm{\psi}-J_h\bm{\psi}\|_0&\leq C (h+(\kappa h)^m)\|u-u_h\|_0,\\
\|\nabla (\varphi-I_h\varphi)\|_0&\leq C (h+(\kappa h)^m)\|u-u_h\|_0.
\end{split}
\label{eq:Ihvarphi1}
\end{equation}
Therefore
\begin{align}
\|\varphi-\varphi_h^-\|_0\leq Ch (h+(\kappa h)^m)\|u-u_h\|_0.\label{eq:regularityvarphi}
\end{align}

Integration by parts,  \eqref{eq:defproj1}, \eqref{eq:projection2}, \eqref{eq:errorB} and \eqref{eq:dualf} reveal that
\begin{align*}
\|u-u_h\|_0^2&=-(u-u_h, \overline{\nabla \cdot\bm{\psi}+i\kappa \varphi})+i\kappa (\overline{\bm{\psi}},\bm{p}-\bm{p}_h)-(\overline{\nabla \varphi}, \bm{p}-\bm{p}_h)\\
&=-b_h(\overline{\bm{\psi}},u-u_h)+i\kappa (u-u_h, \overline{\varphi})+i\kappa (\bm{p}-\bm{p}_h,\overline{\bm{\psi}})+b_h^*(\overline{\varphi}, \bm{p}-\bm{p}_h)+\langle u-u_h, \overline{\varphi}\rangle_{\partial \Omega}\\
&=B_h(\bm{p}-\bm{p}_h,u-u_h;\bm{\psi},\varphi)+i\kappa (u-u_h, \overline{\varphi})+\langle u-u_h, \overline{\varphi}\rangle_{\partial \Omega}\\
&=B_h(\bm{p}-\bm{p}_h,u-u_h;\bm{\psi}-\bm{\psi}_h^-,\varphi-\varphi_h^-)+i\kappa (u-u_h, \overline{\varphi-\varphi_h^-})+\langle u-u_h,\overline{\varphi-\varphi_h^-}\rangle_{\partial \Omega}\\
&=B_h(\bm{p}-\bm{p}_h^+,u-u_h^+;\bm{\psi}-\bm{\psi}_h^-,\varphi-\varphi_h^-)+i\kappa (u-u_h, \overline{\varphi-\varphi_h^-})+\langle u-u_h^+,\overline{\varphi-\varphi_h^-}\rangle_{\partial \Omega}.
\end{align*}
Now we will estimate each of the above term separately. First, the definition of $B_h(\cdot,\cdot;\cdot,\cdot)$ yields
\begin{align*}
&B_h(\bm{p}-\bm{p}_h^+,u-u_h^+;\bm{\psi}-\bm{\psi}_h^-,\varphi-\varphi_h^-)\\
&=i\kappa (\bm{p}-\bm{p}_h^+, \overline{\bm{\psi}-\bm{\psi}_h^-})-b_h^*(u-u_h^+,\overline{\bm{\psi}-\bm{\psi}_h^-})
+b_h(\bm{p}-\bm{p}_h^+,\overline{\varphi-\varphi_h^-}).
\end{align*}
The Cauchy-Schwarz inequality, Lemma~\ref{lemma:uh+}, \eqref{eq:Ihvarphi} and \eqref{eq:Ihvarphi1} imply
\begin{align*}
i\kappa (\overline{\bm{p}-\bm{p}_h^+}, \bm{\psi}-\bm{\psi}_h^-)&\leq \kappa \|\bm{p}-\bm{p}_h^+\|_0\|\bm{\psi}-\bm{\psi}_h^-\|_0\\
&\leq C (h+(\kappa h)^m)\|u-u_h\|_0\|\bm{p}-J_h\bm{p}\|_0.
\end{align*}
We have from the Cauchy-Schwarz inequality, the inverse inequality, Lemma~\ref{lemma:uh+} and \eqref{eq:Ihvarphi}
\begin{align*}
b_h^*(u-u_h^+,\overline{\bm{\psi}-\bm{\psi}_h^-})&\leq C \Big(\|\bm{\psi}-J_h\bm{\psi}\|_0\|\nabla (u-I_hu)\|_0+\|I_hu-u_h^+\|_Z\|J_h\bm{\psi}-\bm{\psi}_h^-\|_0\Big)\\
&\leq C\Big(\|\bm{\psi}-J_h\bm{\psi}\|_0\|\nabla (u-I_hu)\|_0+h^{-1}\|I_hu-u_h^+\|_0\|J_h\bm{\psi}-\bm{\psi}_h^-\|_0\Big)\\
&\leq C \Big(h+(\kappa h)^m\Big)\Big(\|\nabla (u-I_hu)\|_0+\|\bm{p}-J_h\bm{p}\|_0\Big)\|u-u_h\|_0.
\end{align*}
Similarly, we can get
\begin{align*}
b_h(\bm{p}-\bm{p}_h^+,\overline{\varphi-\varphi_h^-})&\leq C\Big( \|\bm{p}-J_h\bm{p}\|_0\|\nabla (\varphi-I_h\varphi)\|_0+\|I_h\varphi-\varphi_h^-\|_Z\|J_h\bm{p}-\bm{p}_h^+\|_0\Big)\\
&\leq C (h+(\kappa h)^m)\|\bm{p}-J_h\bm{p}\|_0\|u-u_h\|_0.
\end{align*}
An appeal to the Cauchy-Schwarz inequality and \eqref{eq:regularityvarphi} yields
\begin{align*}
i\kappa (u-u_h, \overline{\varphi-\varphi_h^-})&\leq C \kappa \|u-u_h\|_0\|\varphi-\varphi_h^-\|_0\\
&\leq C \kappa h \|u-u_h\|_0^2(h+(\kappa h)^m).
\end{align*}
The trace inequality, the inverse inequality, Lemma~\ref{lemma:uh+}, \eqref{eq:Ihvarphi1} and \eqref{eq:regularityvarphi} imply
\begin{align*}
\langle u-u_h^+,\overline{\varphi-\varphi_h^-}\rangle_{\partial \Omega}&\leq \|u-u_h^+\|_{0,\partial \Omega}\|\varphi-\varphi_h^-\|_{0,\partial \Omega}\\
&\leq C( h^{-\frac{1}{2}}\|u-u_h^+\|_0+h^{\frac{1}{2}}\|\nabla(u-u_h^+)\|_0)(h^{-\frac{1}{2}}\|\varphi-\varphi_h^-\|_0
+h^{\frac{1}{2}}\|\nabla (\varphi-\varphi_h^-)\|_0)\\
&\leq C \Big(h\|\bm{p}-J_h\bm{p}\|_0+\|u-I_hu\|_0+h\|\nabla (u-I_hu)\|_0\Big)\Big(h+(\kappa h)^m\Big)\|u-u_h\|_0.
\end{align*}
Combining the preceding estimates yields \eqref{eq:L2u} under the condition that $\kappa h$ is sufficiently small.

Next, we estimate $\|\bm{p}-\bm{p}_h\|_0$. We have from \eqref{eq:errorJhp}
%\begin{align*}
%\kappa \|\bm{p}-\bm{p}_h\|_0^2&=-\textnormal{Im} A_h(\bm{p}-\bm{p}_h,u-u_h;\bm{p}-\bm{p}_h,u-u_h)\\
%&=-\textnormal{Im} (A_h(\bm{p}-\bm{p}_h,u-u_h;\bm{p}-\bm{p}_h,u-u_h)+i\kappa (u-u_h,\overline{u-u_h})\\
%&\;+\langle u-u_h,\overline{u-u_h}\rangle_{\partial \Omega})+\kappa\|u-u_h\|_0^2\\
%&=-\textnormal{Im} (A_h(\bm{p}-\bm{p}_h,u-u_h;\bm{p}-\bm{p}_h^-,u-u_h^-)+i\kappa (u-u_h,\overline{u-u_h^-})\\
%&\;+\langle u-u_h, u-u_h^-\rangle_{\partial \Omega})+\kappa \|u-u_h\|_0^2\\
%&=-\textnormal{Im} (A_h(\bm{p}-\bm{p}_h^+,u-u_h^+;\bm{p}-\bm{p}_h^-,u-u_h^-)+i\kappa (u-u_h,\overline{u-u_h^-})\\
%&\;+\langle u-u_h^+, u-u_h^-\rangle_{\partial \Omega})+\kappa \|u-u_h\|_0^2\\
%&\leq
%\end{align*}
\begin{align*}
&A_h(J_h\bm{p}-\bm{p}_h, I_hu-u_h;J_h\bm{p}-\bm{p}_h, I_hu-u_h)+i\kappa (I_hu-u_h, \overline{I_hu-u_h})+\langle I_hu-u_h, \overline{I_hu-u_h}\rangle_{\partial \Omega}\\
&=-i\kappa (\overline{J_h\bm{p}-\bm{p}},J_h\bm{p}-\bm{p}_h)+i\kappa (I_hu-u,\overline{I_hu-u_h}).
\end{align*}
Taking the imaginary parts of the above equation implies
\begin{align*}
\kappa \|J_h\bm{p}-\bm{p}_h\|_0^2\leq C \Big(\kappa \|J_h\bm{p}-\bm{p}\|_0^2+\kappa \|I_hu-u\|_0^2+\kappa\|I_hu-u_h\|_0^2\Big).
\end{align*}
Consequently, we have
\begin{align*}
\|J_h\bm{p}-\bm{p}_h\|_0\leq C \Big(\|J_h\bm{p}-\bm{p}\|_0+\|I_hu-u\|_0+\|I_hu-u_h\|_0\Big).
\end{align*}
Thus, the proof is complete.

%\begin{align*}
%&-A_h(J_h\bm{p}-\bm{p}_h,I_hu-u_h;J_h\bm{\psi}-\bm{\psi}_h^+,I_h\varphi-\varphi_h^+)=
%-A_h(J_h\bm{p}-\bm{p}_h^+,I_hu-u_h^+;J_h\bm{\psi}-\bm{\psi}_h^+,I_h\varphi-\varphi_h^+)\\
%&\;-\langle I_hu-u_h^+,I_h\varphi-\varphi_h^+\rangle_{\partial \Omega}-i\kappa (I_hu-u, I_h\varphi-\varphi_h^+\rangle_{\partial \Omega}+i\kappa (I_hu-u_h, I_h\varphi-\varphi_h^+)\\
%&\;+\langle I_hu-u_h,I_h\varphi-\varphi_h^+\rangle_{\partial \Omega}\\
%&=i\kappa (\overline{J_h\bm{p}-\bm{p}},\varphi_h^+)-i\kappa (I_hu-u, I_h\varphi-\varphi_h^+\rangle_{\partial \Omega}+i\kappa (I_hu-u_h, I_h\varphi-\varphi_h^+)\\
%&\;+\langle I_hu-u_h,I_h\varphi-\varphi_h^+\rangle_{\partial \Omega}
%\end{align*}

\end{proof}

Combining Lemmas~\ref{lemma:approximation} and \ref{lemma:L2}, we can get the following estimates.
\begin{corollary}\label{corollary:stability}

If $u\in H^2(\Omega)$, then
\begin{align*}
\|u-u_h\|_0&\leq C \Big(h^2+(\kappa h)^{2m}\Big)M_{f,g},\\
\|\bm{p}-\bm{p}_h\|_0&\leq C \Big(h+(\kappa h)^m+(\kappa h)^{2m}\Big)M_{f,g}.
\end{align*}

\end{corollary}

Combining Lemmas~\ref{lemma:appro1} and \ref{lemma:L2}, we have
\begin{theorem} If $(\bm{p},u)\in H^m(\Omega)^2\times H^{m+1}(\Omega)$ satisfying
\begin{align*}
\kappa\|\bm{p}\|_m+\|u\|_{m+1}\leq C \kappa^m,
\end{align*}
then
\begin{align*}
\|u-u_h\|_0&\leq C\frac{1}{\kappa}\Big((\kappa h)^{m+1}+\kappa(\kappa h)^{2m}\Big),\\
\|\bm{p}-\bm{p}_h\|_0&\leq C\Big((\kappa h)^m+(\kappa h)^{2m}\Big).
\end{align*}

\end{theorem}

By combining Lemma~\ref{lemma:ubound} and Corollary~\ref{corollary:stability}, we can obtain the following stability estimates for staggered DG method.
\begin{corollary}(stability)
Suppose the solution $u\in H^2(\Omega)$, then the following estimate holds provided $\kappa h$ is sufficiently small
\begin{align*}
\|u_h\|_0+\|\bm{p}_h\|_0&\leq C M_{f,g}.
\end{align*}

\end{corollary}

\section{Numerical experiments}\label{sec:numerical}

This section presents numerical experiments for the validation
of the theoretical results and investigates the accuracy of our staggered DG method. In addition, an singular example is given to test the capability of the proposed method to capture the singularities.

\begin{example}(Typical smooth solution example)
\end{example}

We first consider a Helmholtz equation defined on the unit square $\Omega=[-0.5,0.5]\times [-0.5,0.5]$. Here, we set $f=\frac{\sin(\kappa r)}{r}$ in \eqref{eq:model1}-\eqref{eq:model2} and $g$ is chosen such that the exact solution is given by
\begin{align*}
u=\frac{\cos(\kappa r)}{\kappa}-\frac{\cos(\kappa)+i\sin(\kappa)}{\kappa(J_0(\kappa)+iJ_1(\kappa)}J_0(\kappa r)
\end{align*}
in polar coordinates, where $J_{\nu}(z)$ are Bessel functions of the first kind.

The exact solution for $\kappa=50$ is displayed in Figure~\ref{exact}, and the numerical solution for $\kappa=50$ by using $P^1$ and $P^3$ is shown in Figure~\ref{numerical}. As expected, higher order polynomial approximation yields better numerical solution.

\begin{figure}[H]
\centering
\includegraphics[width=7cm]{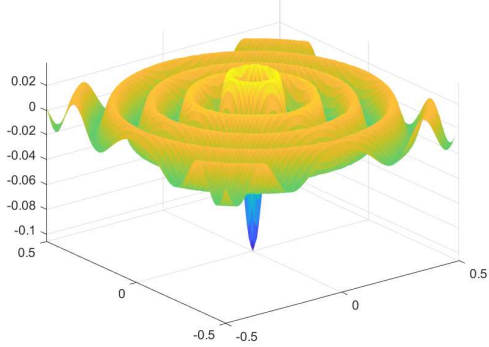}
\caption{Exact solution for $\kappa=50$.}
\label{exact}
\end{figure}

\begin{figure}[H]
\centering
\scalebox{0.3}{
\includegraphics[width=20cm]{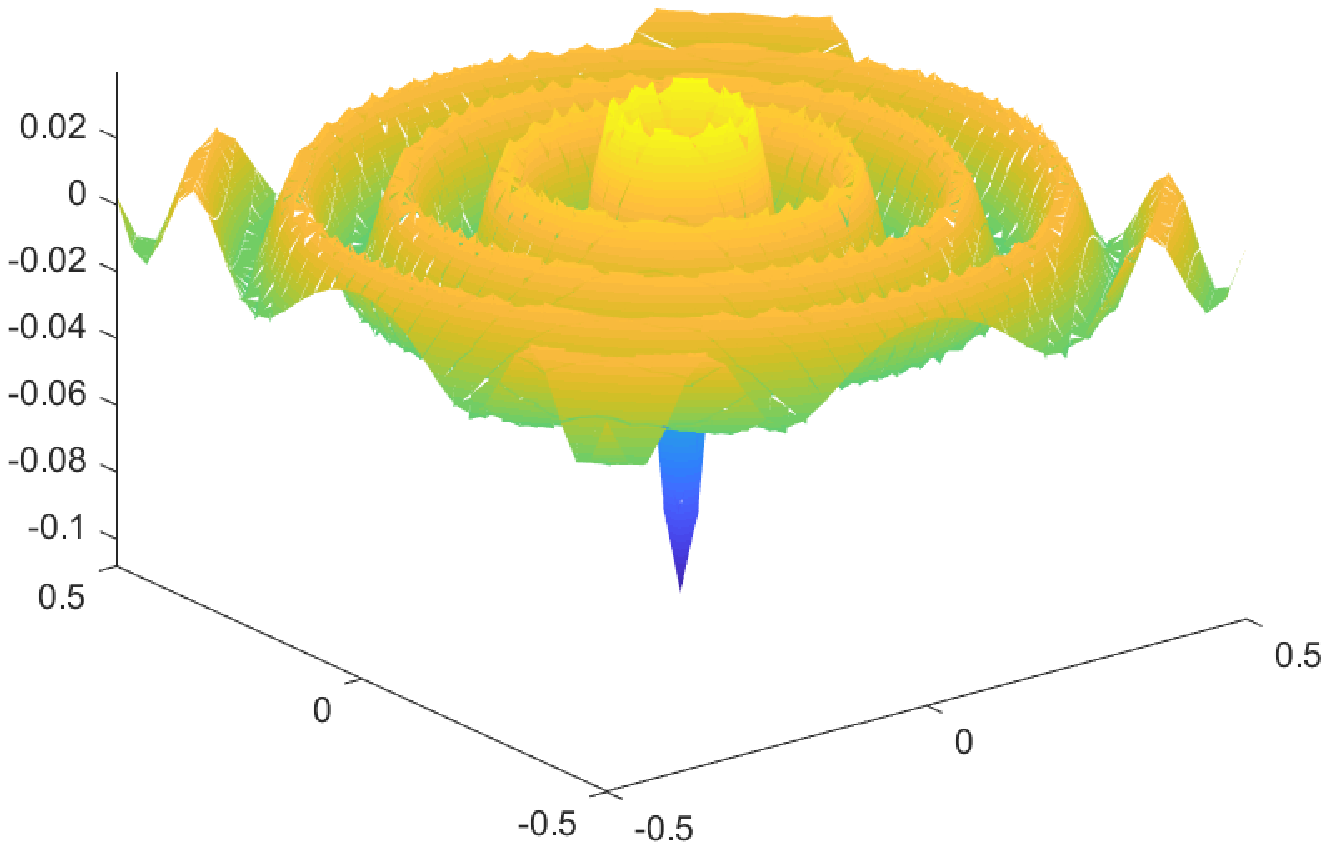}
}
\scalebox{0.3}{
\includegraphics[width=20cm]{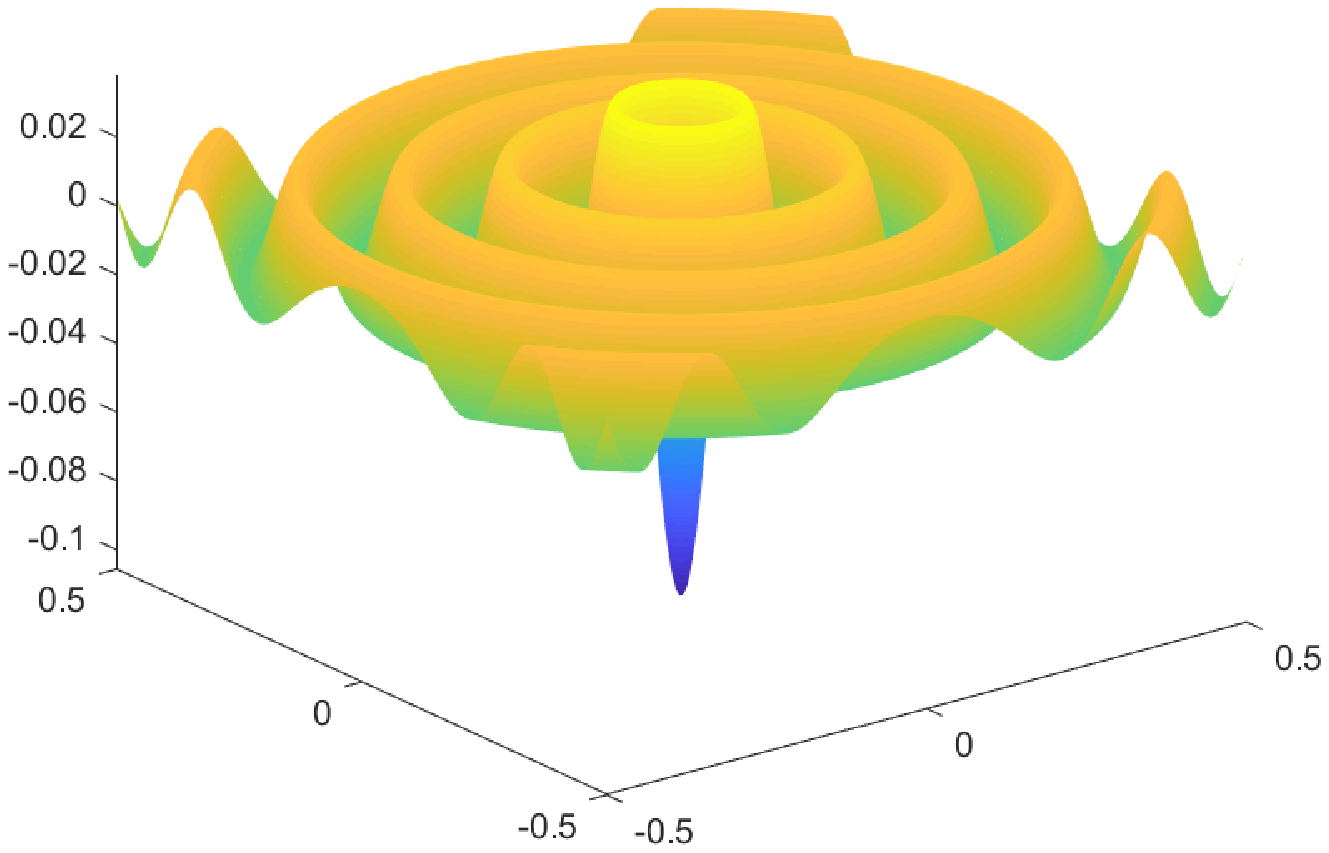}
}
\caption{Numerical solution for $\kappa=50$ by $P_1$ (left) and $P_3$ (right).}
\label{numerical}
\end{figure}

For the fixed wave number $\kappa$,we first show the dependence of the convergence of $\|u-u_h\|_0$, $\|\bm{p}-\bm{p}_h\|_0$ on polynomial order $m$ and mesh size $h$. The convergence history against the number of degrees of freedom for $\kappa=50$ and $\kappa=100$ with different polynomial orders on square grids are reported in Figure~\ref{solution-square}. It is easy to see that the pollution errors always appear on the coarse meshes, and the optimal convergence can be obtained on fine meshes for different polynomial orders, which confirm the proposed theories.

\begin{figure}[H]\label{solution-square}
 \centering
    \begin{minipage}[b]{0.4\textwidth}
      \includegraphics[width=1\textwidth]{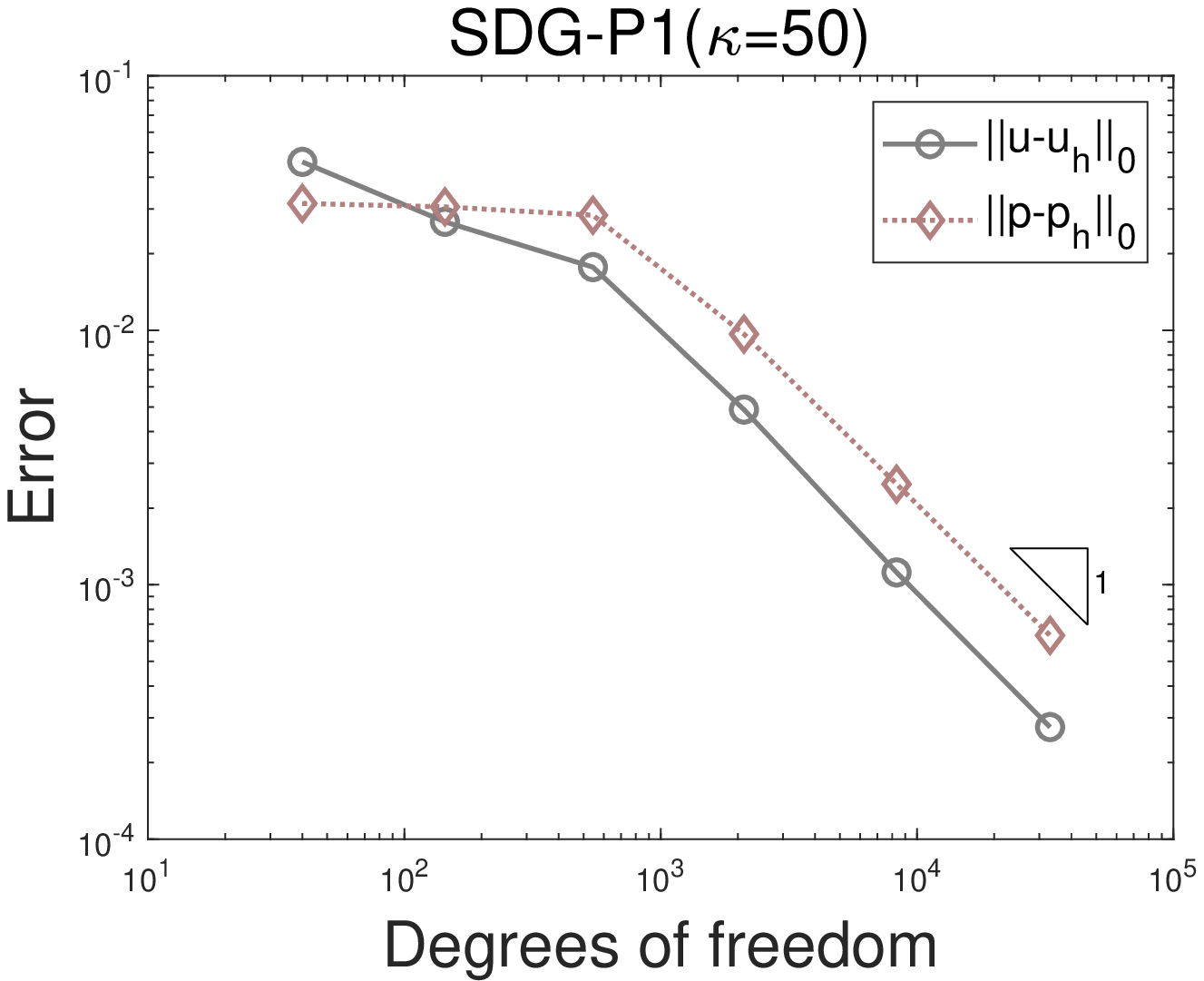}
    \end{minipage}%
    \begin{minipage}[b]{0.4\textwidth}
      \includegraphics[width=1\textwidth]{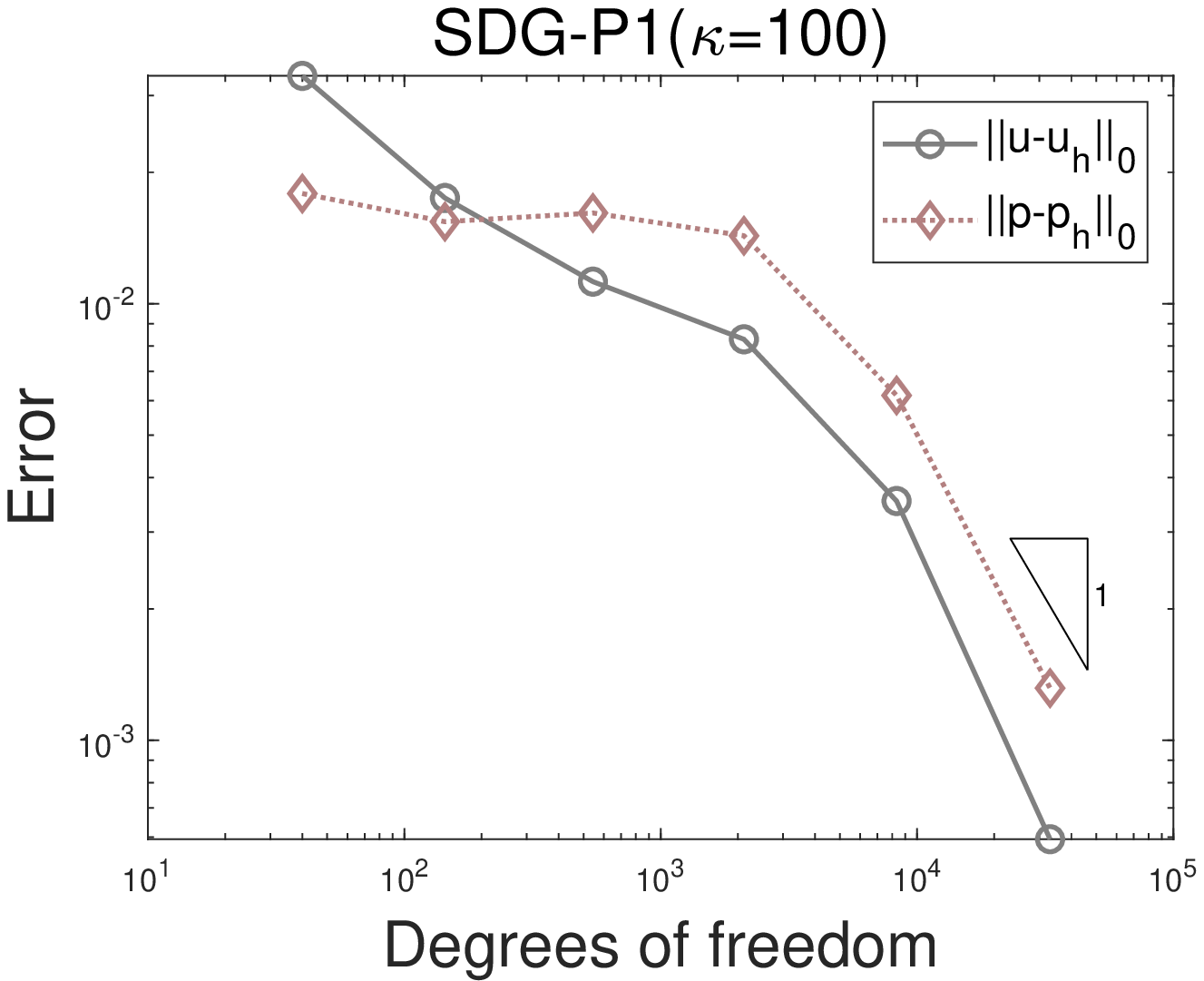}
    \end{minipage}
    \begin{minipage}[b]{0.4\textwidth}
      \includegraphics[width=1\textwidth]{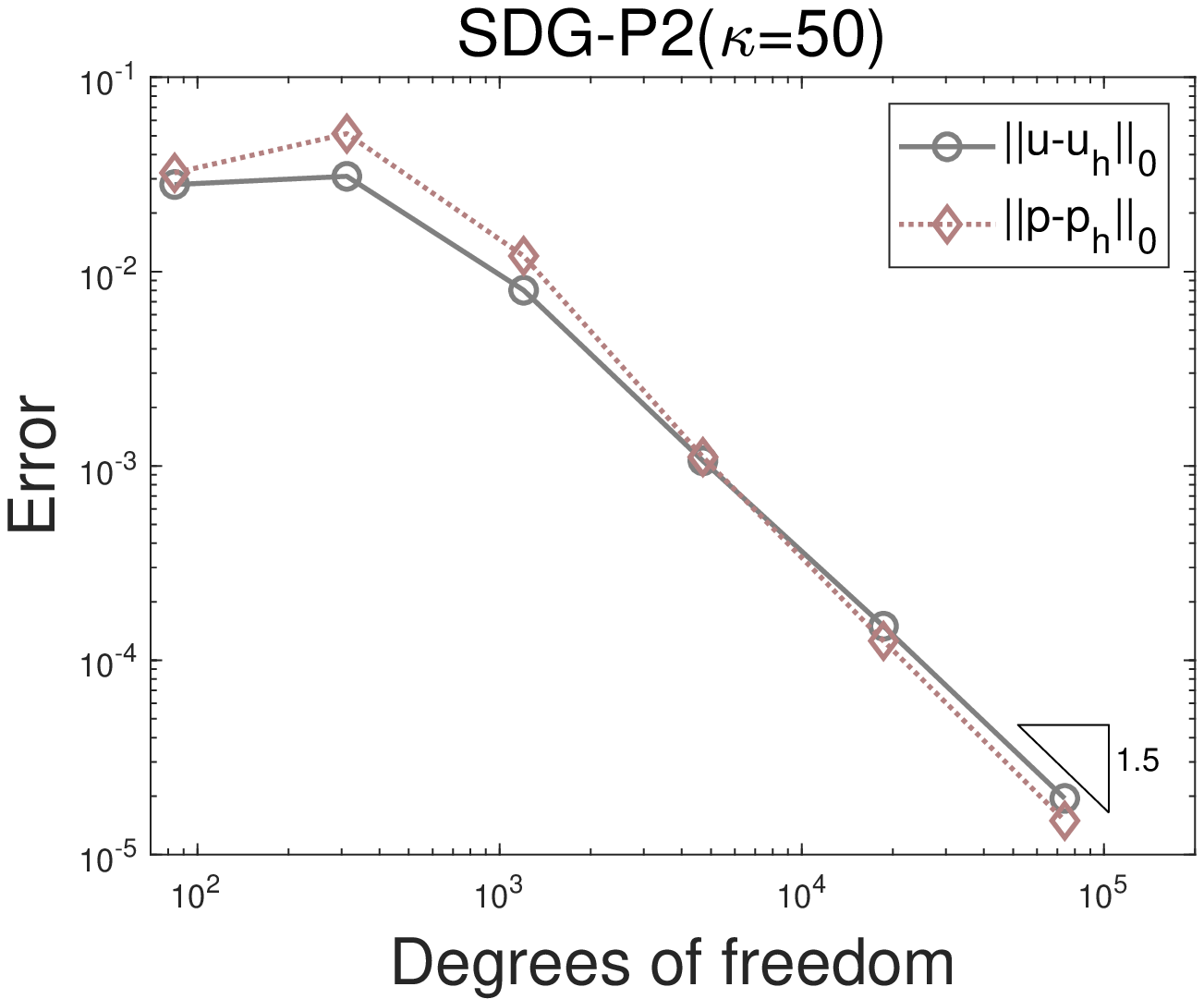}
    \end{minipage}
     \begin{minipage}[b]{0.4\textwidth}
      \includegraphics[width=1\textwidth]{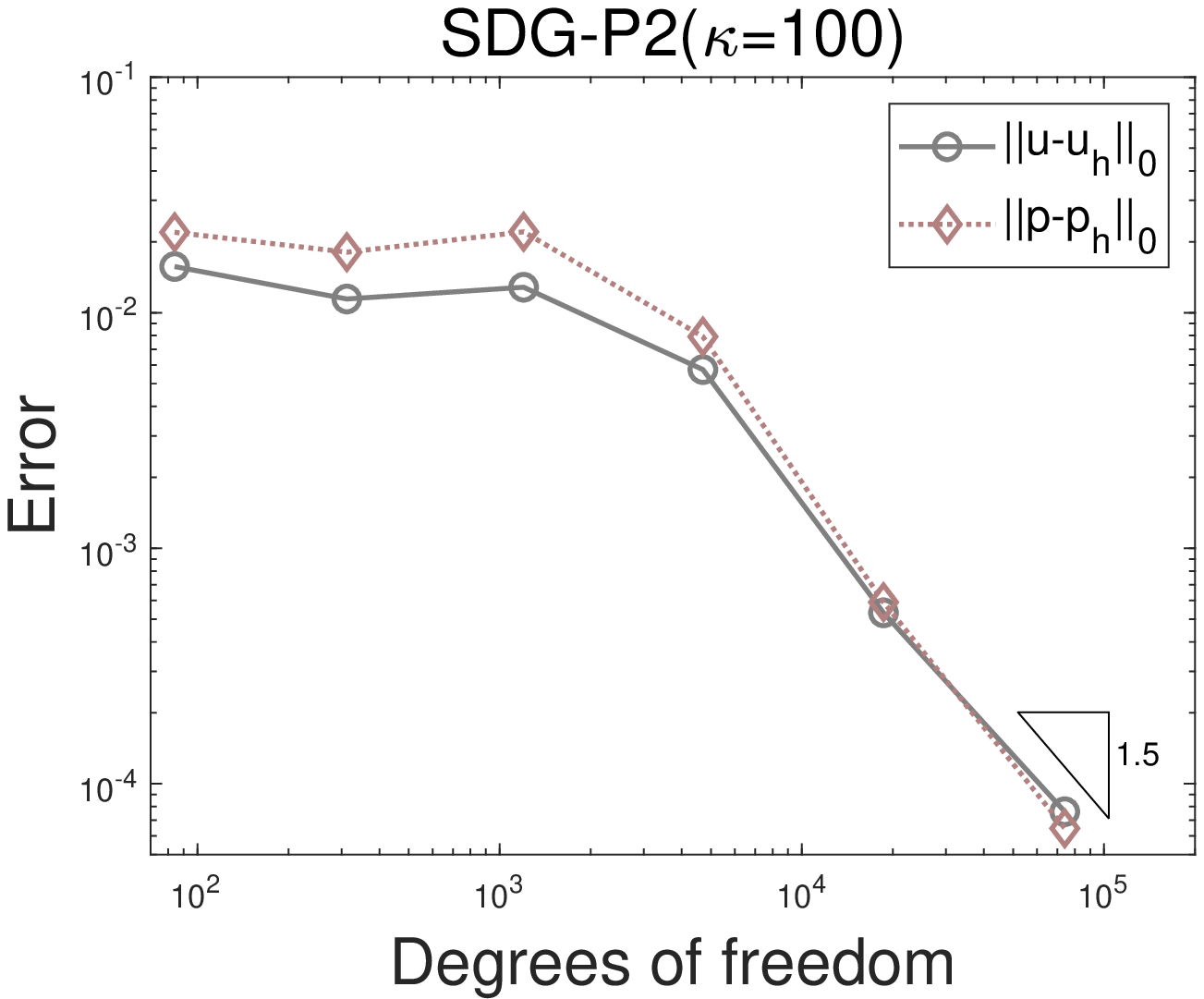}
    \end{minipage}
     \begin{minipage}[b]{0.4\textwidth}
      \includegraphics[width=1\textwidth]{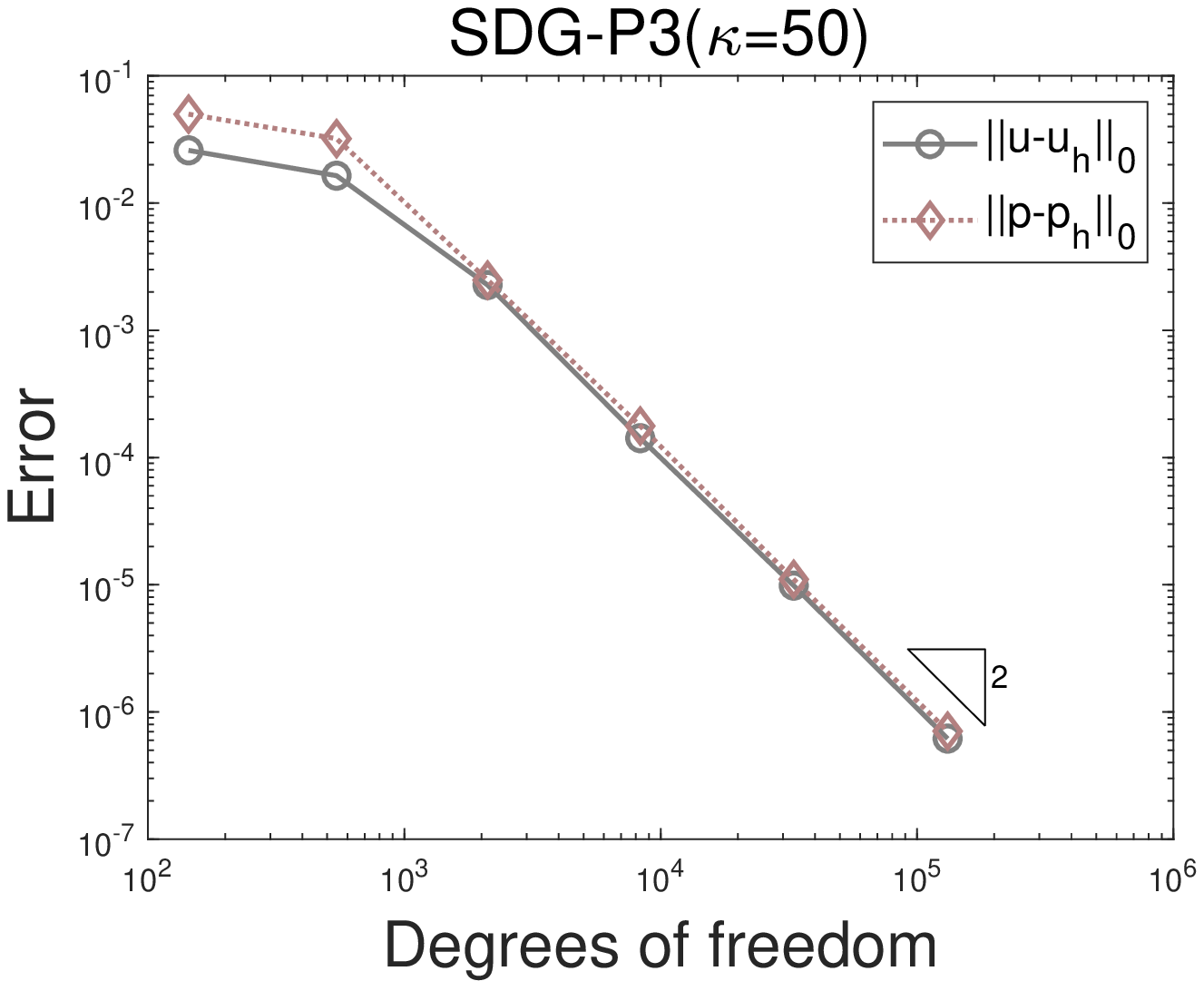}
    \end{minipage}
    \begin{minipage}[b]{0.4\textwidth}
      \includegraphics[width=1\textwidth]{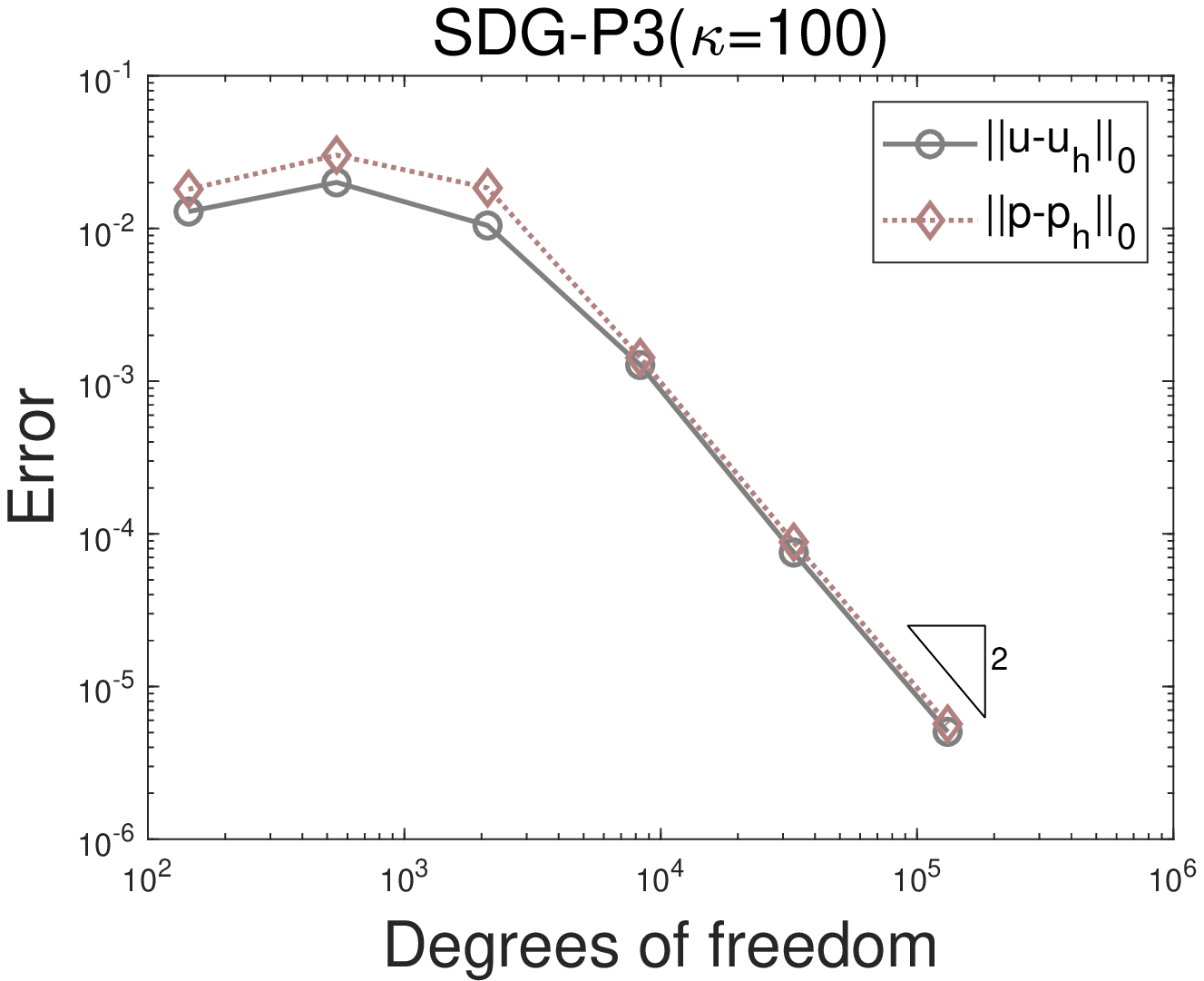}
    \end{minipage}%
  \caption{Errors of $\|u-u_h\|_0$,$\|\bm{p}-\bm{p}_h\|_0$ for $\kappa=50$ by $\mbox{SDG}-P^{1}$,$\mbox{SDG}-P^{2}$ and $\mbox{SDG}-P^{3}$ approximations (left, top to bottom). Errors of $\|u-u_h\|_0$,$\|\bm{p}-\bm{p}_h\|_0$ for $\kappa=100$ by $\mbox{SDG}-P^{1}$,$\mbox{SDG}-P^{2}$ and $\mbox{SDG}-P^{3}$ approximations (right, top to bottom).}
\end{figure}

Next we verify the convergence properties of the SDG method for different wave numbers by piecewise $P^1,P^2$ and $P^3$ approximations, respectively. We can see from Figure~\ref{figure:con} that in the pre-asymptotic region, the errors always oscillate for different polynomial orders, and optimal convergence can be obtained for fine meshes. In addition, high order polynomial approximation has better performances.
\begin{figure}[H]\label{figure:con}
 \centering
    \begin{minipage}[b]{0.4\textwidth}
      \includegraphics[width=1\textwidth]{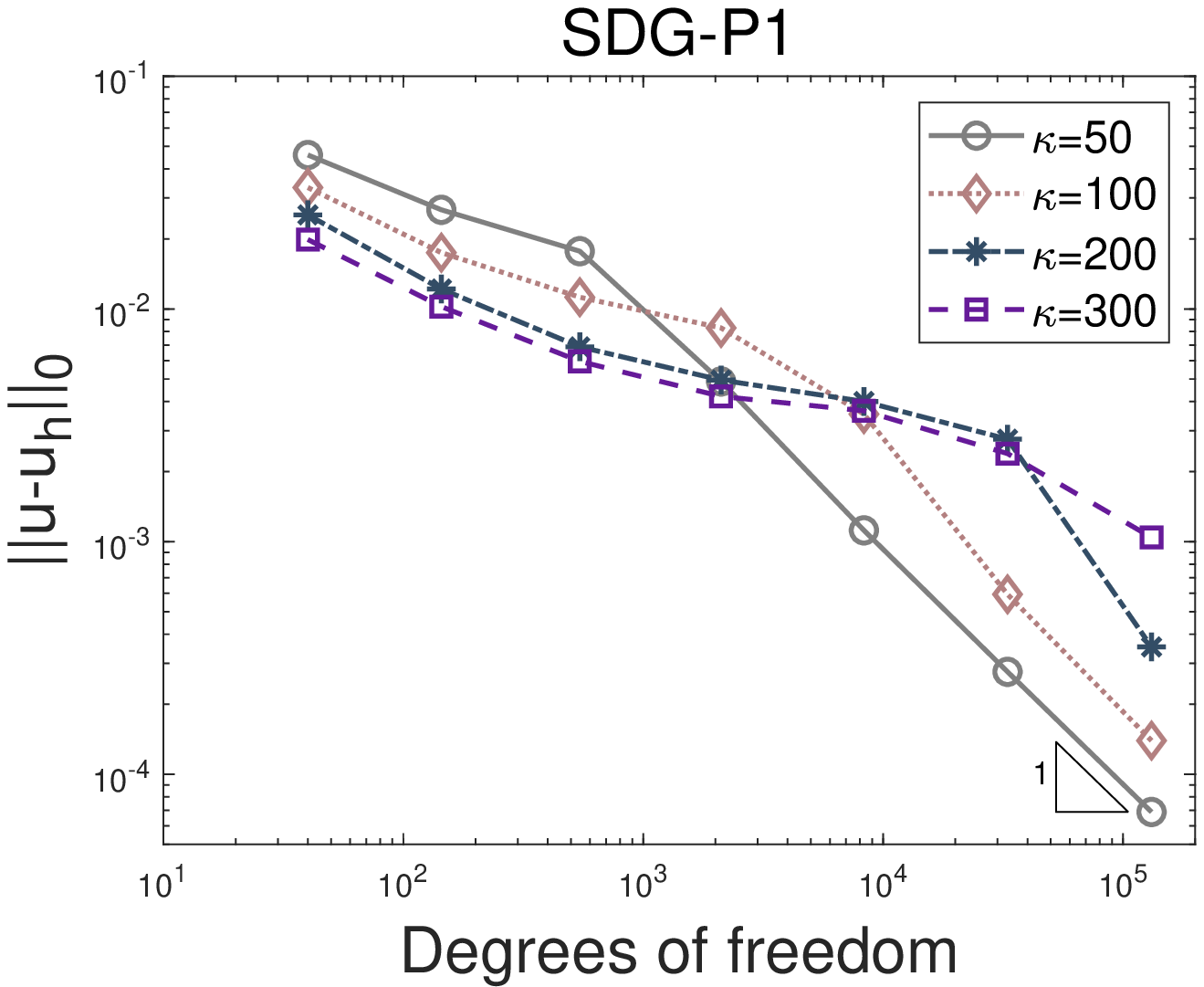}
    \end{minipage}%
    \begin{minipage}[b]{0.4\textwidth}
      \includegraphics[width=1\textwidth]{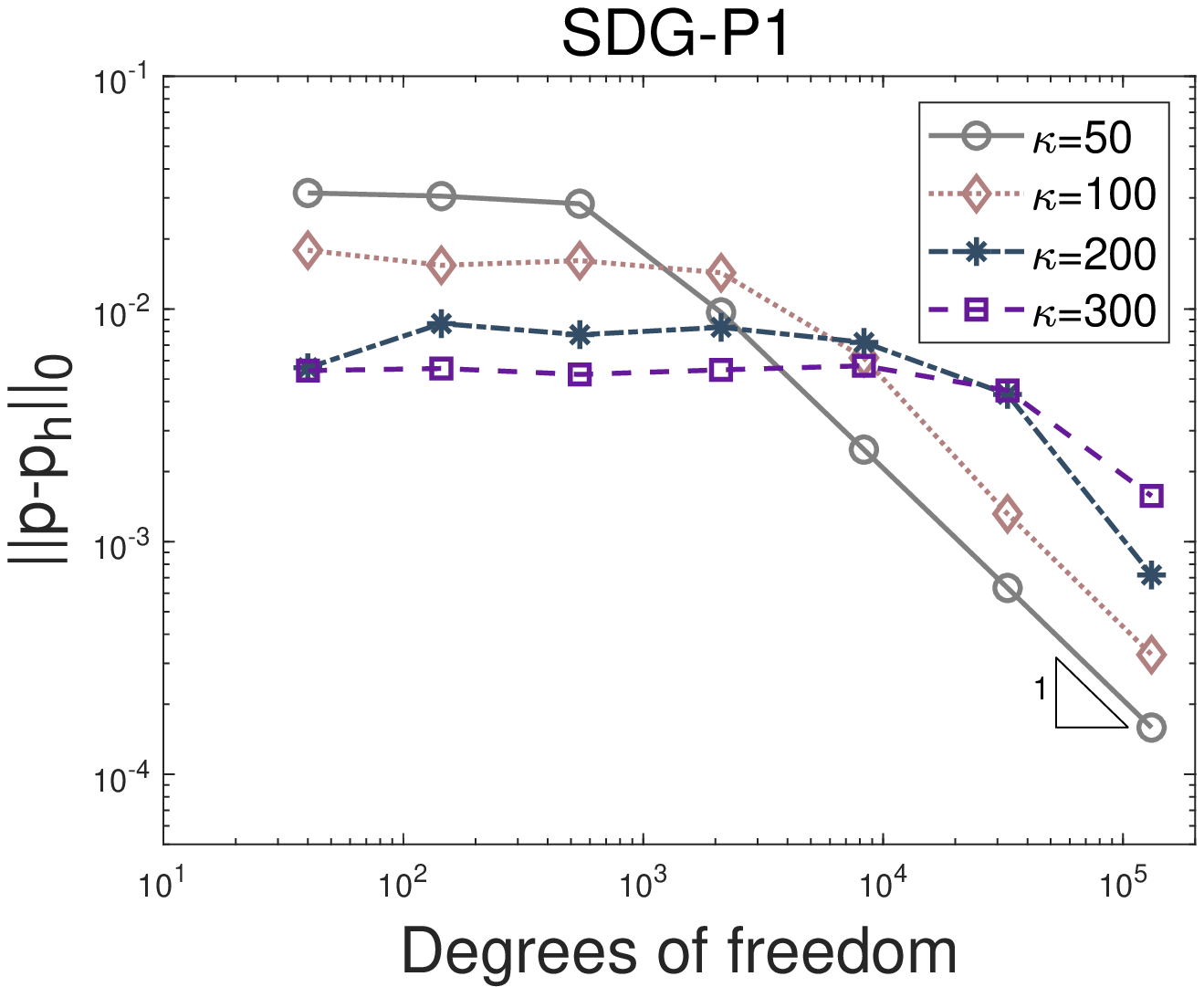}
    \end{minipage}
    \begin{minipage}[b]{0.4\textwidth}
      \includegraphics[width=1\textwidth]{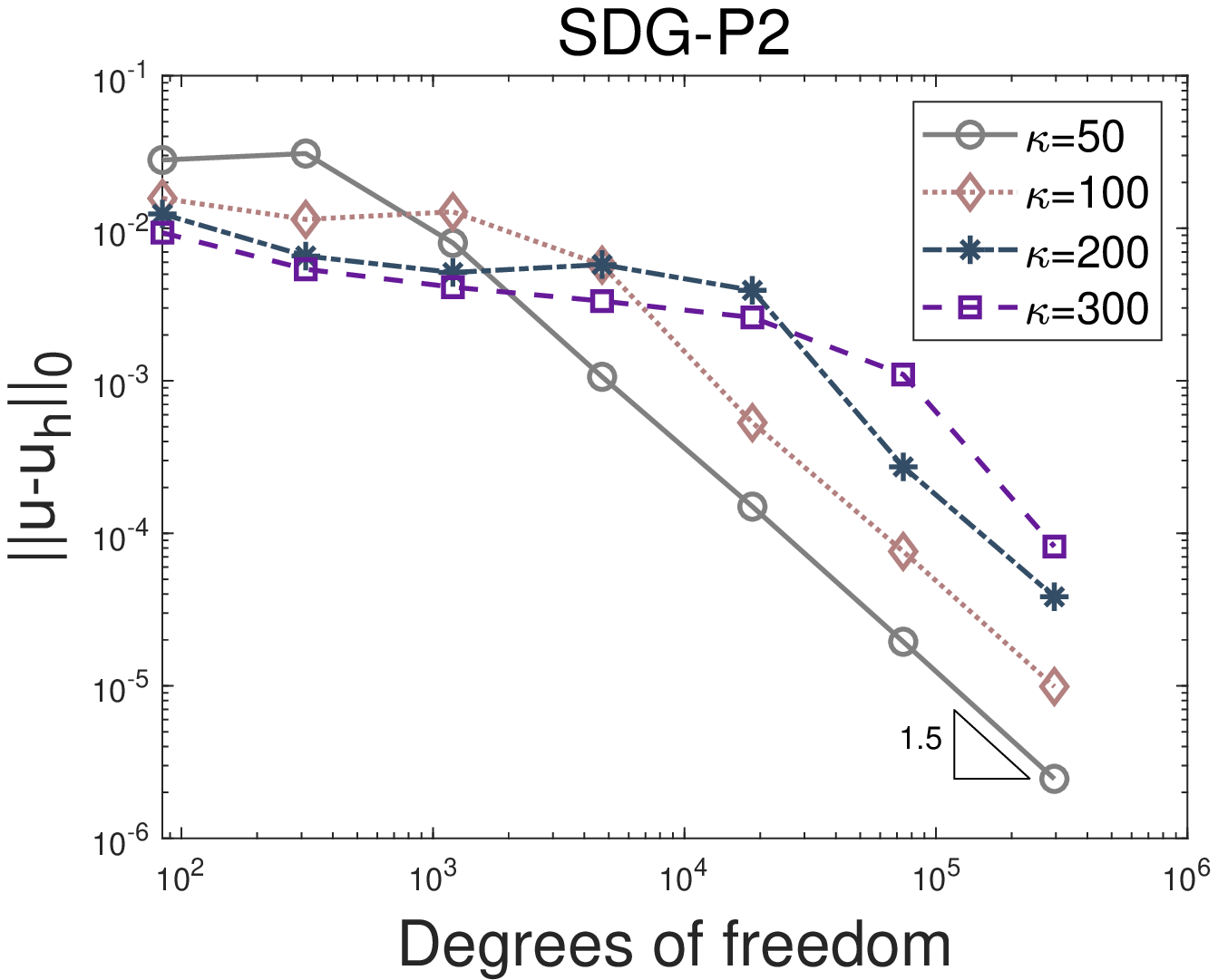}
    \end{minipage}
     \begin{minipage}[b]{0.4\textwidth}
      \includegraphics[width=1\textwidth]{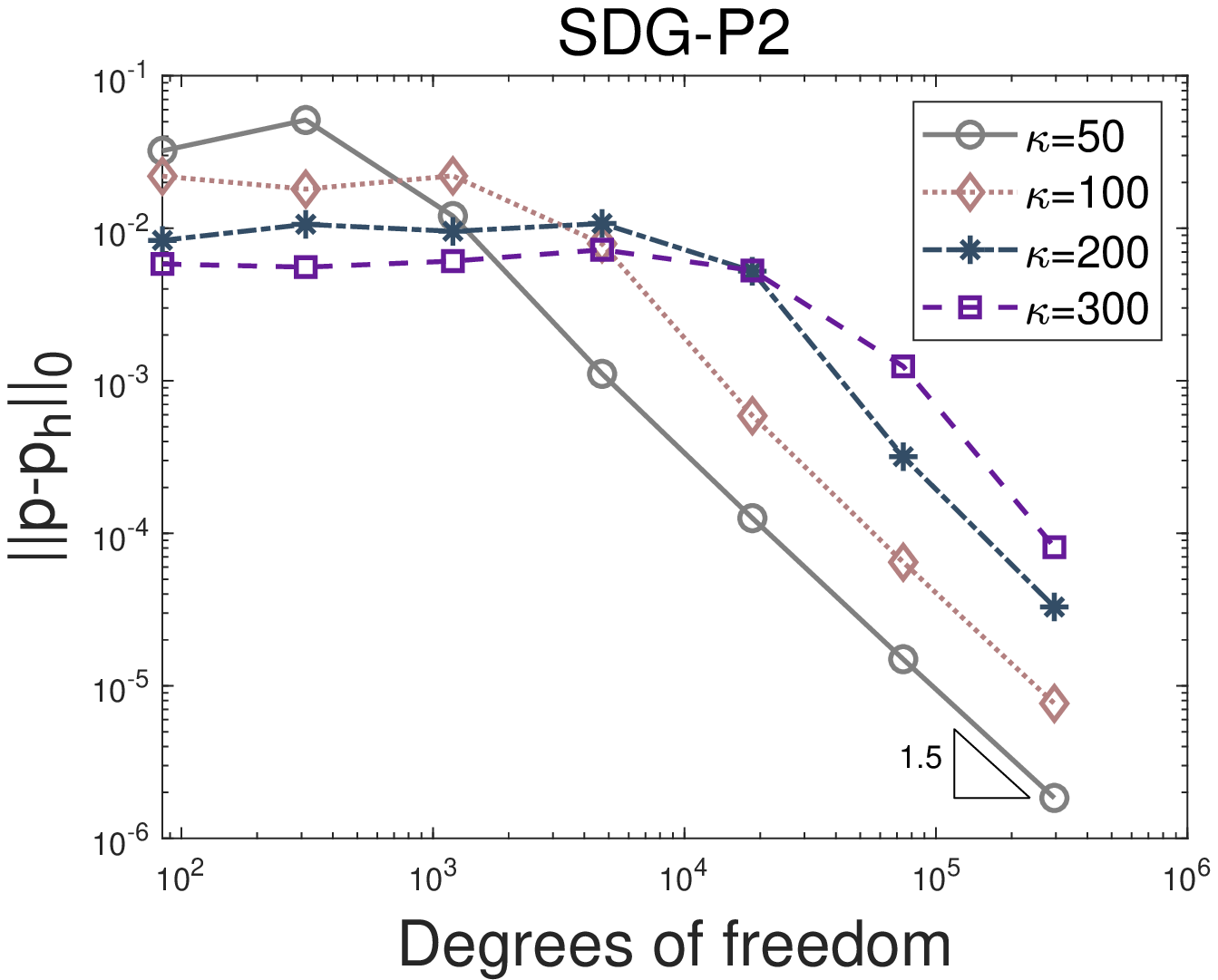}
    \end{minipage}
     \begin{minipage}[b]{0.4\textwidth}
      \includegraphics[width=1\textwidth]{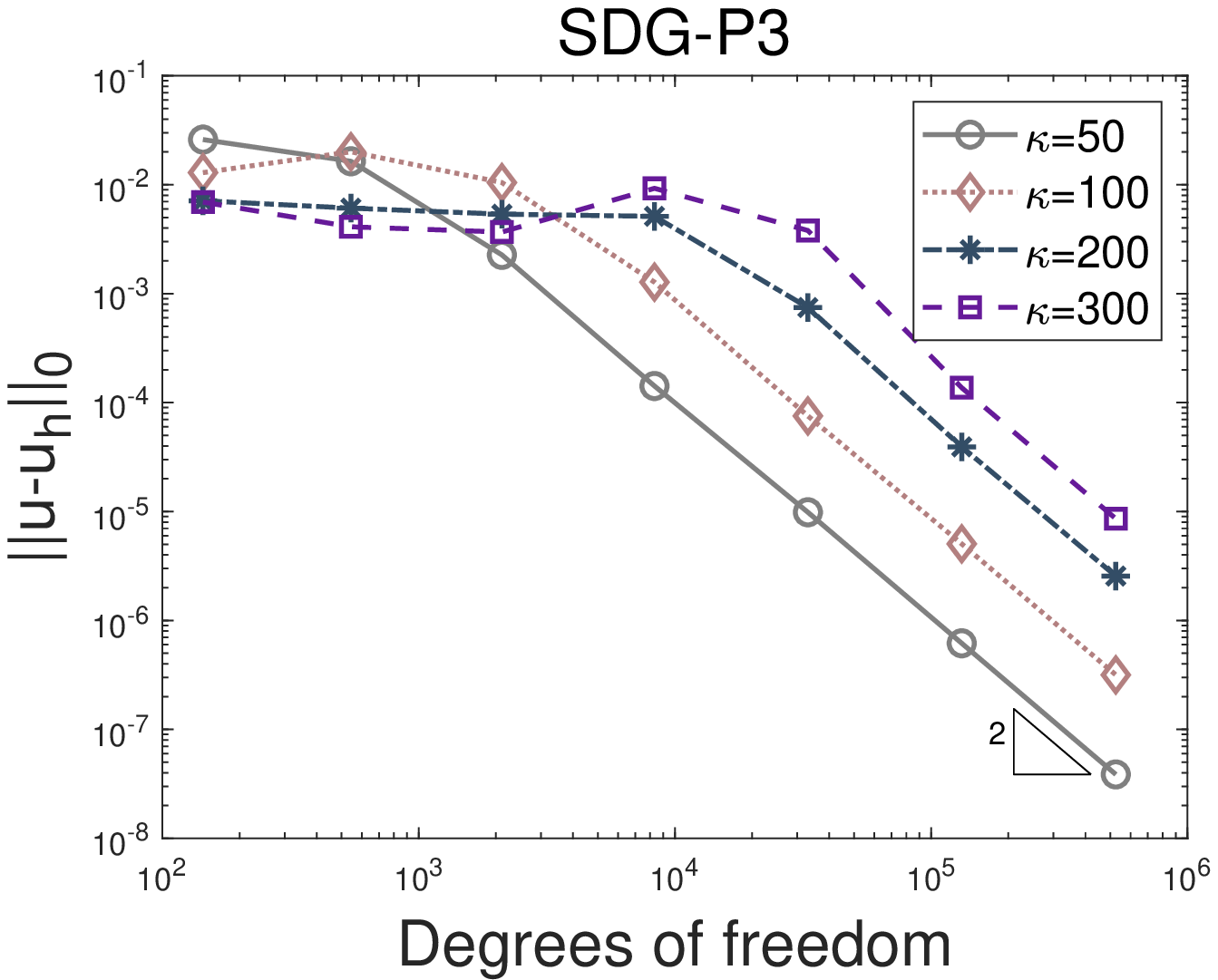}
    \end{minipage}
    \begin{minipage}[b]{0.4\textwidth}
      \includegraphics[width=1\textwidth]{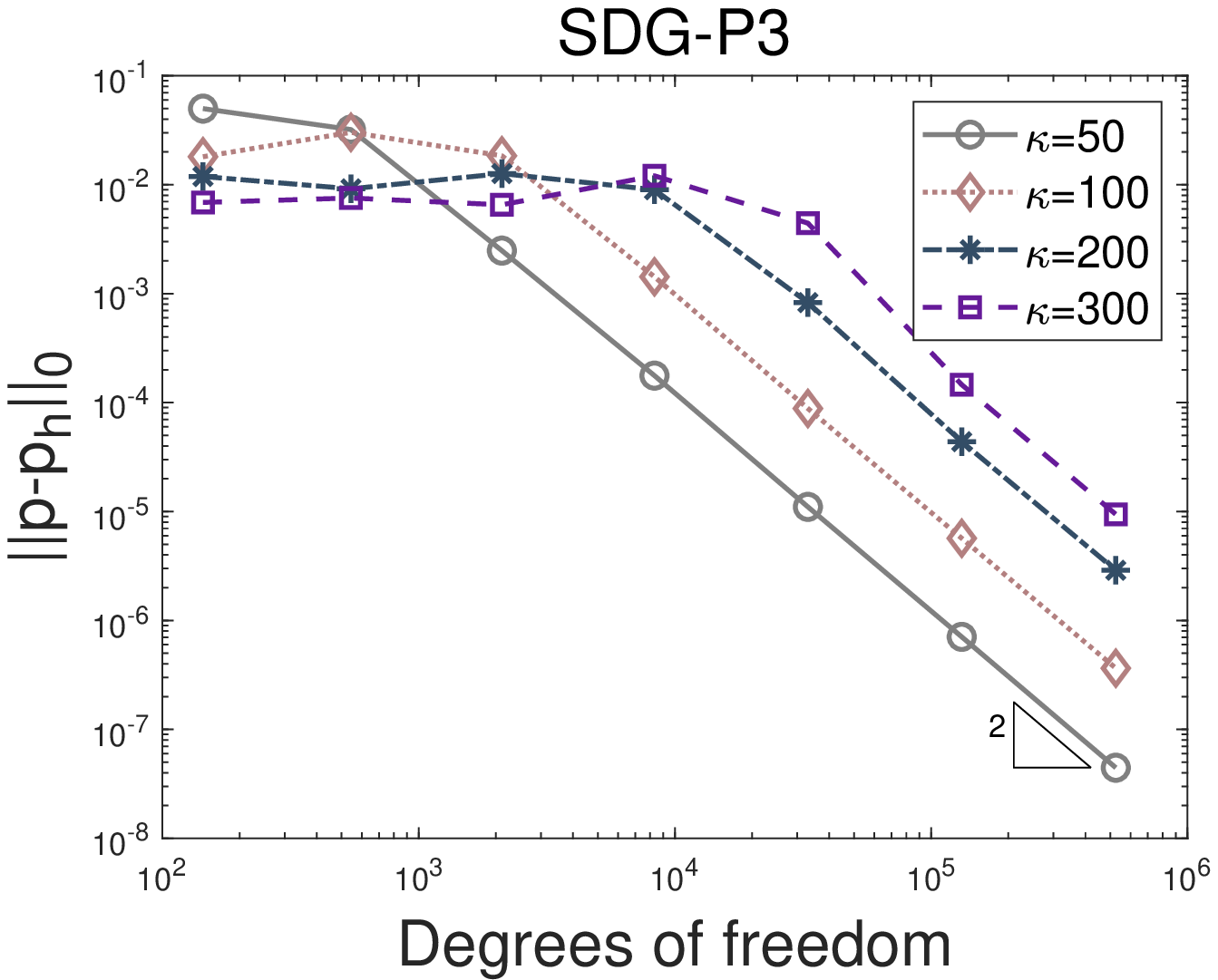}
    \end{minipage}%
  \caption{The errors of $\|u-u_h\|_0$ for $\kappa=50,100,200,300$ by $\mbox{SDG}-P^{1}$,$\mbox{SDG}-P^{2}$ and $\mbox{SDG}-P^{3}$ (left, top to bottom). The errors of $\|\bm{p}-\bm{p}_h\|_0$ for $\kappa=50,100,200,300$ by $\mbox{SDG}-P^{1}$,$\mbox{SDG}-P^{2}$ and $\mbox{SDG}-P^{3}$ (right, top to bottom).}
\end{figure}

%\begin{figure}[H]
%\centering
%\scalebox{0.3}{
%\includegraphics[width=20cm]{solution_waveuh.eps}
%}
%\scalebox{0.3}{
%\includegraphics[width=20cm]{solution_waveph.eps}
%}
%\caption{.}
%\label{fixed}
%\end{figure}
%
%\begin{figure}[H]
%\centering
%\scalebox{0.3}{
%\includegraphics[width=20cm]{solution_waveuhP2.eps}
%}
%\scalebox{0.3}{
%\includegraphics[width=20cm]{solution_wavephP2.eps}
%}
%\caption{.}
%\end{figure}
%
%\begin{figure}[H]
%\centering
%\scalebox{0.3}{
%\includegraphics[width=20cm]{solution_waveuhP3.eps}
%}
%\scalebox{0.3}{
%\includegraphics[width=20cm]{solution_wavephP3.eps}
%}
%\caption{.}
%\end{figure}

%\begin{figure}[H]
%\centering
%\scalebox{0.3}{
%\includegraphics[width=20cm]{errorfixedkh.eps}
%}
%\scalebox{0.3}{
%\includegraphics[width=20cm]{errorfixedkhP2.eps}
%}
%\caption{Errors of $\|u-u_h\|_0$, $\|\bm{p}-\bm{p}_h\|_0$ for fixed $\kappa h=3.125$ for $P_1$ (left) and $P_2$ (right).}
%\end{figure}

In addition, to test the flexibility of the proposed method on rough grids, we employ the $h$-perturbation grids (cf. \cite{LinaPark}) as shown in Figure~\ref{distort}. The $L^2$ errors for both the scalar and vector variables by using different polynomial approximations for $\kappa=50$ and $\kappa=10$ are displayed in Figure~\ref{solution-perturbation}. From which we can see that similar performances as for square grids can be obtained, thus, the proposed method is robust in the sense that it can be flexibly applied to rough grids. Undisplayed numerical experiments on polygonal grids also yield similar results.
\begin{figure}[H]
\centering
\includegraphics[width=6cm]{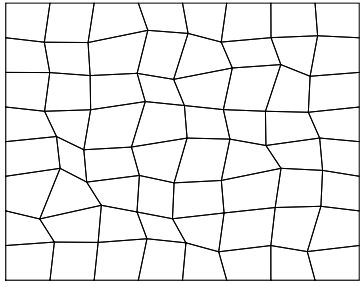}
\caption{A distorted rectangular mesh.}
\label{distort}
\end{figure}

\begin{figure}[H]\label{solution-perturbation}
    \centering
    \begin{minipage}[b]{0.4\textwidth}
      \includegraphics[width=1\textwidth]{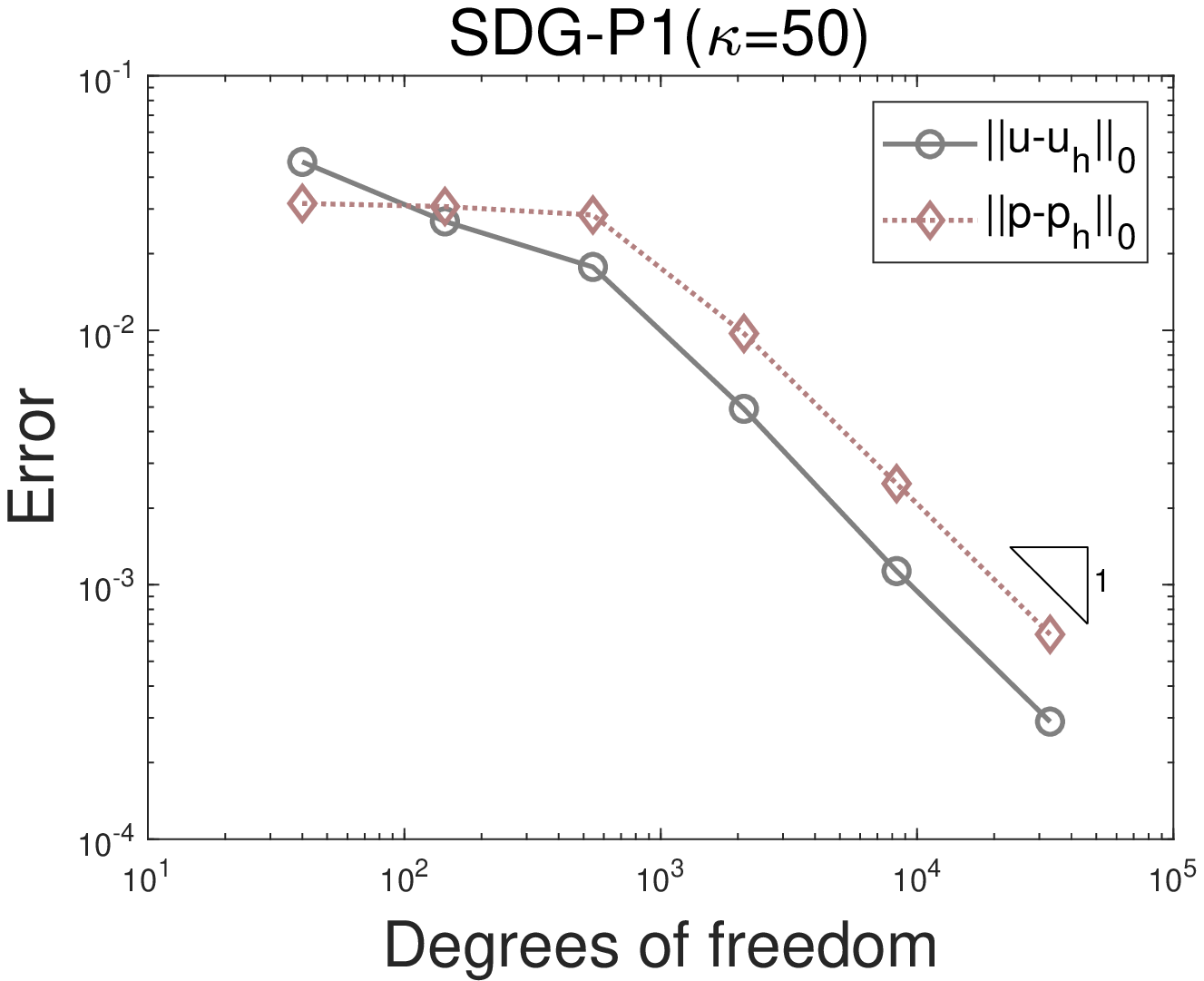}
    \end{minipage}%
    \begin{minipage}[b]{0.4\textwidth}
      \includegraphics[width=1\textwidth]{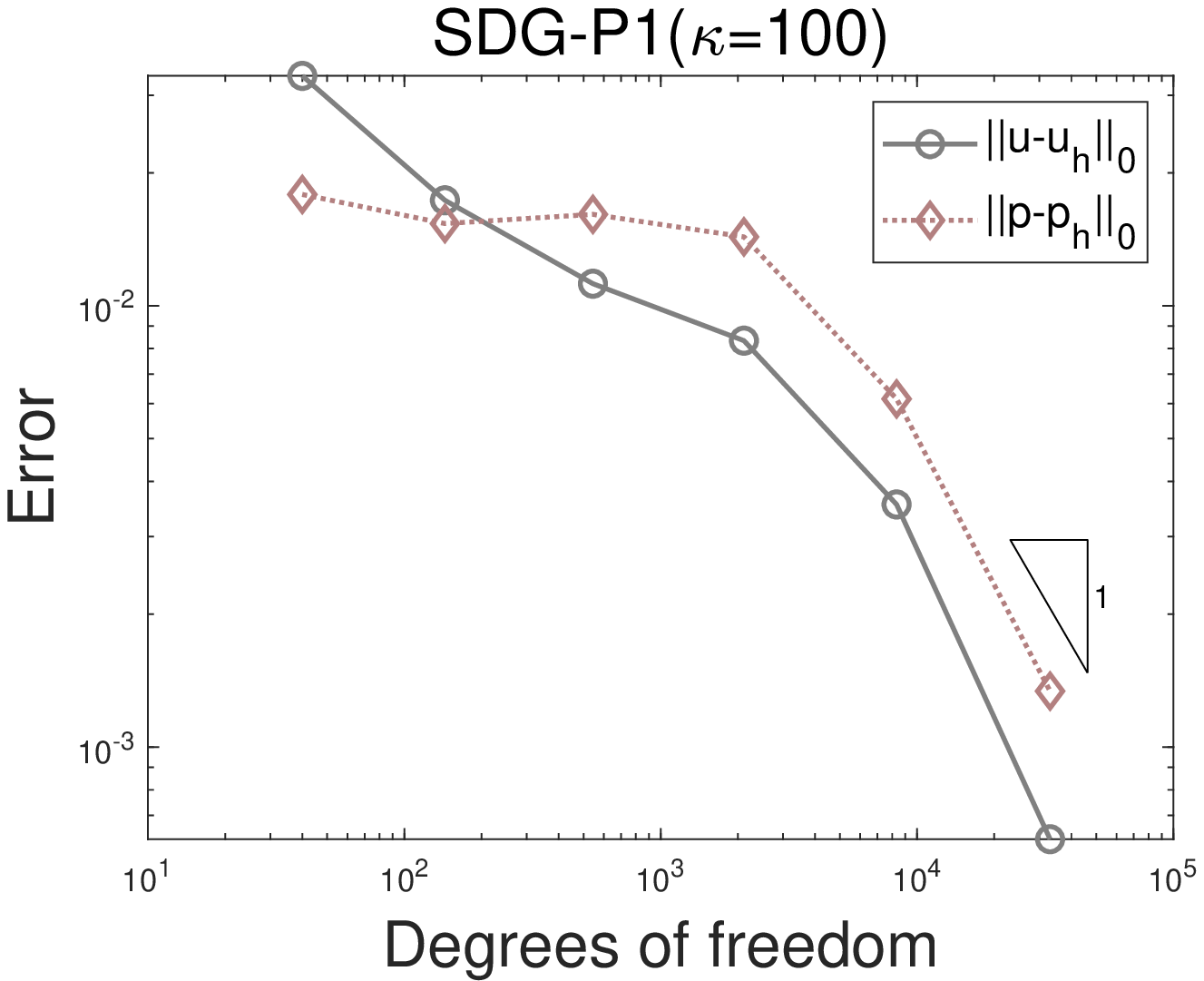}
    \end{minipage}
    \begin{minipage}[b]{0.4\textwidth}
      \includegraphics[width=1\textwidth]{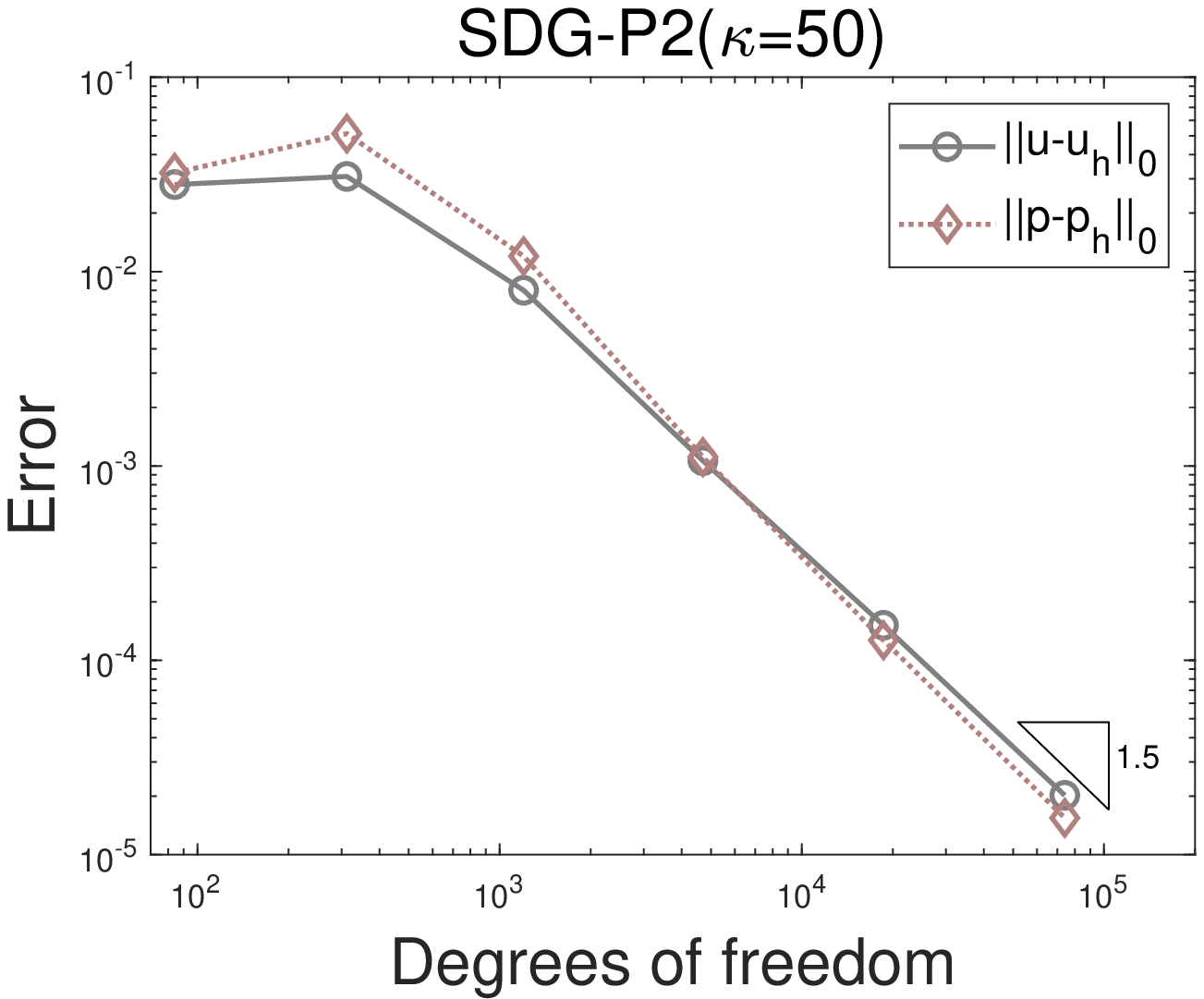}
    \end{minipage}
     \begin{minipage}[b]{0.4\textwidth}
      \includegraphics[width=1\textwidth]{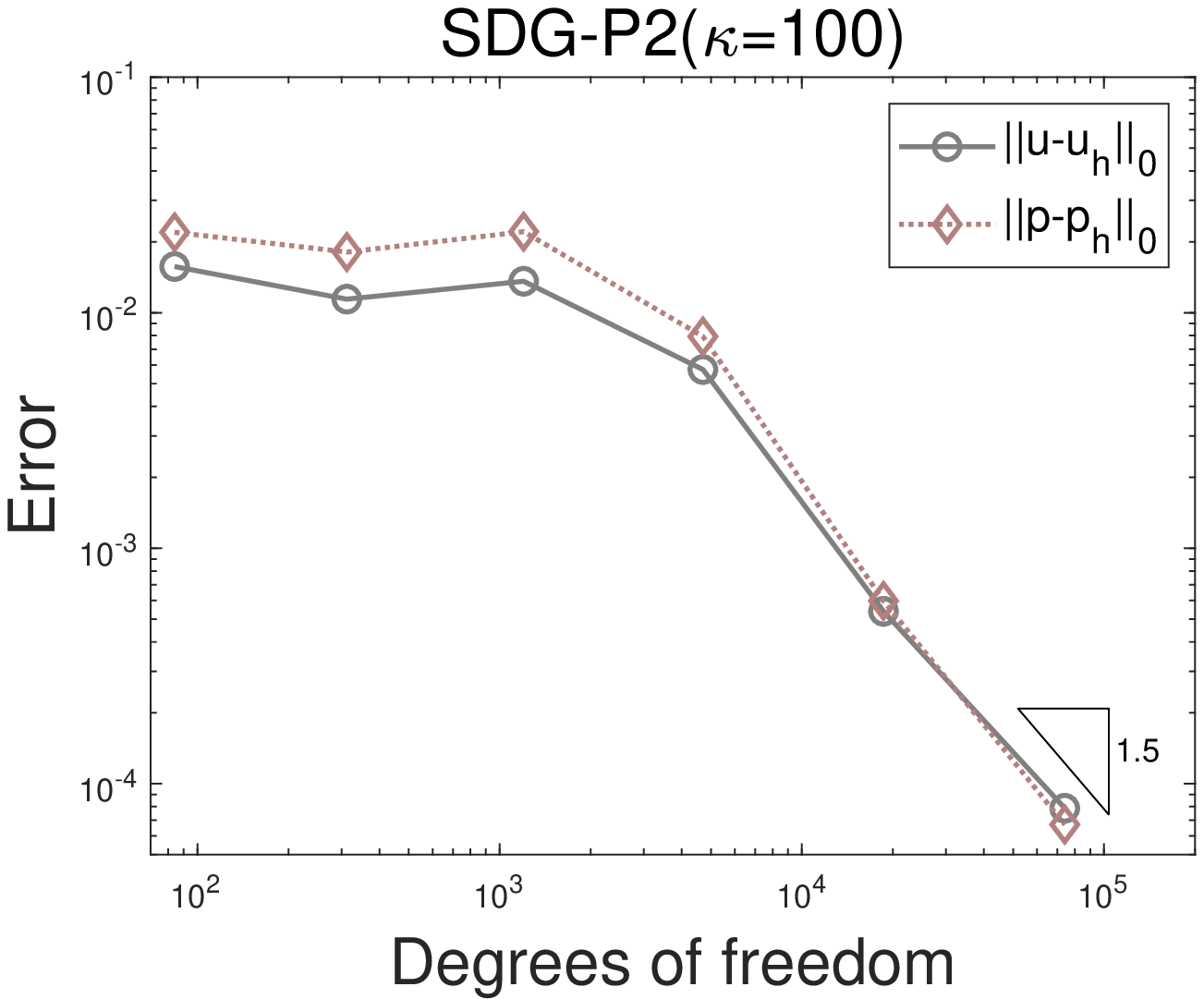}
    \end{minipage}
     \begin{minipage}[b]{0.4\textwidth}
      \includegraphics[width=1\textwidth]{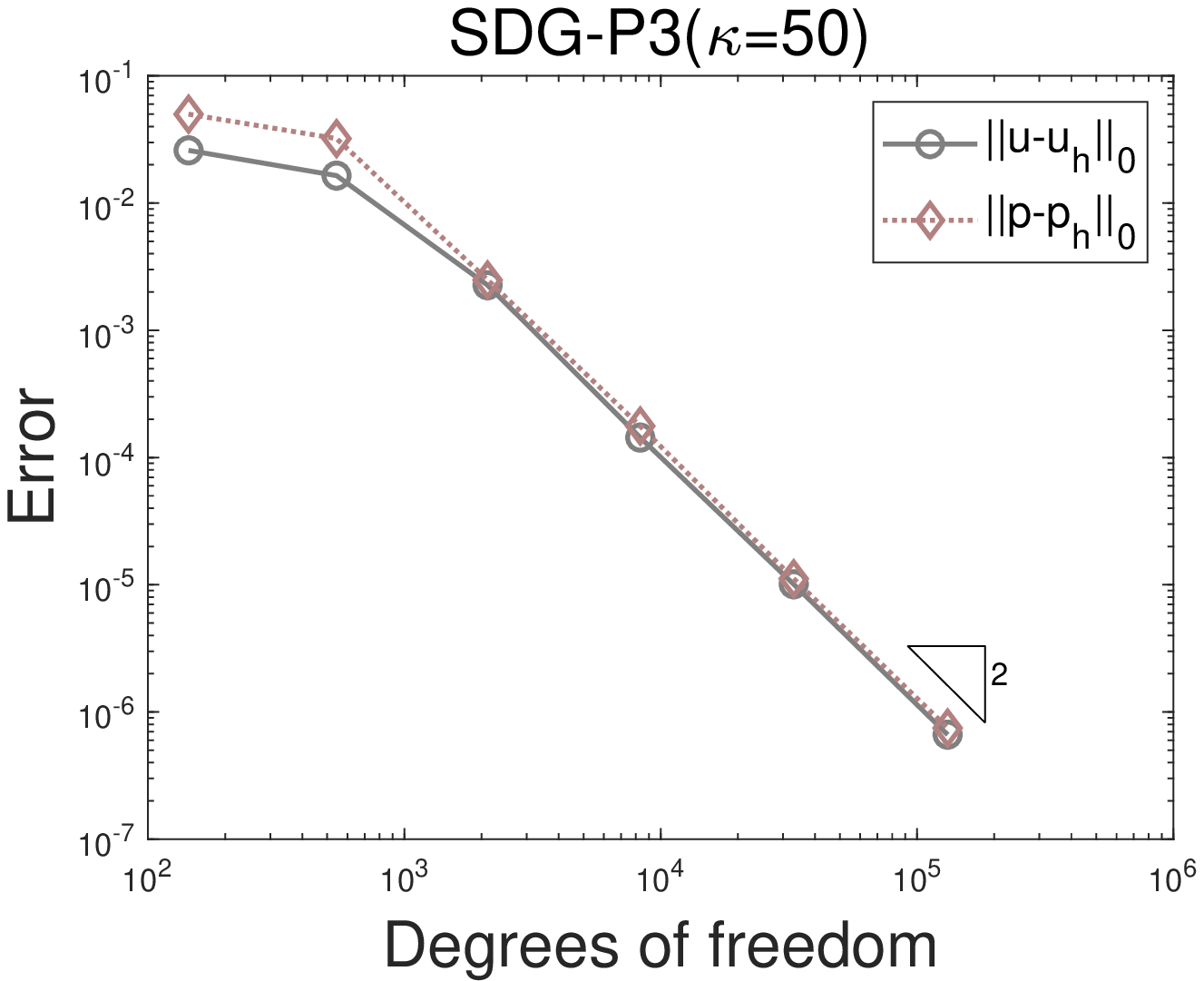}
    \end{minipage}
    \begin{minipage}[b]{0.4\textwidth}
      \includegraphics[width=1\textwidth]{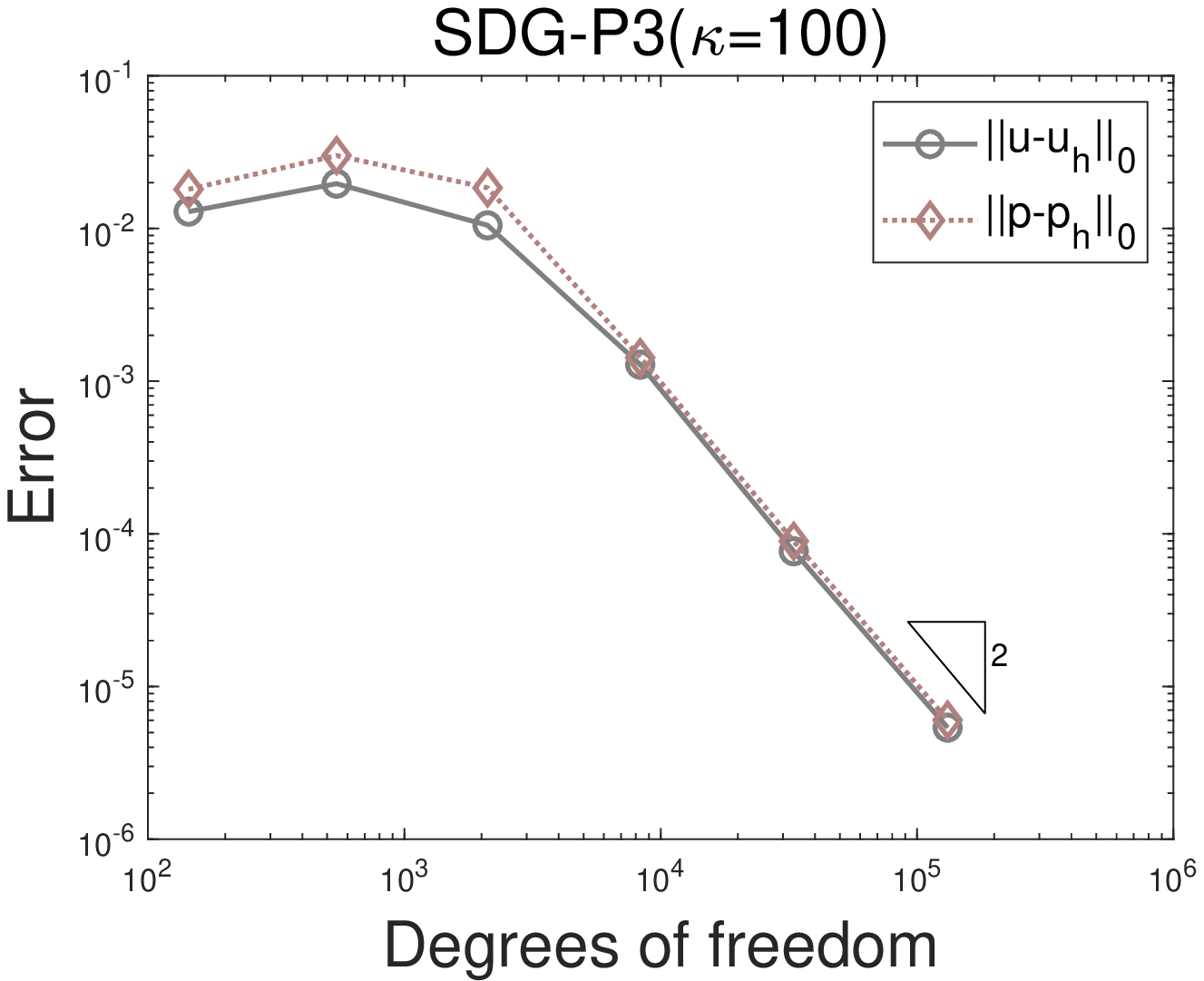}
    \end{minipage}%
  \caption{Errors of $\|u-u_h\|_0$,$\|\bm{p}-\bm{p}_h\|_0$ for $\kappa=50$ by $\mbox{SDG}-P^{1}$,$\mbox{SDG}-P^{2}$ and $\mbox{SDG}-P^{3}$ approximations (left, top to bottom). Errors of $\|u-u_h\|_0$,$\|\bm{p}-\bm{p}_h\|_0$ for $\kappa=100$ by $\mbox{SDG}-P^{1}$,$\mbox{SDG}-P^{2}$ and $\mbox{SDG}-P^{3}$ approximations (right, top to bottom).}
\end{figure}

\begin{example}(Singular solution example)

\end{example}

In this example, we consider a square domain $\Omega=(0,1)\times (-0.5,0.5)$. The boundary conditions $g$ and $f$ are chosen such that the exact solution is given by
\begin{align*}
u=J_\xi(\kappa r)\cos(\xi \theta)
\end{align*}
for $\xi=1$, $\xi=2/3$ and $\xi=3/2$, respectively, where $J_\xi$ denotes the Bessel function of the first kind and order $\xi$. It is well known that $u$ is smooth for $\xi\in \mathbb{N}$, while its derivative has a singularity at origin for $\xi\neq \mathbb{N}$ (cf. \cite{Hiptmair11}). We fix $\kappa=10$, and we compute the numerical solution in the regular case $\xi=1$ and in the singular cases $\xi=1/2$ and $\xi=3/2$. The profiles of the numerical solutions corresponding to these three cases are displayed in Figures~\ref{ex2-solution1} and \ref{ex2-solution2}. The convergence order against the mesh size $h$ for these three cases by using $P^1$ and $P^2$ polynomial approximations are reported in Tables~\ref{table:1}--\ref{table:3}. It is easy to see that when $\xi=1$, optimal convergence can be obtained for both variables in $L^2$ norm using linear and quadratic approximations. when $\xi=3/2$, a reduced convergence order can be achieved for both the flux variable and the scalar variable, in addition, the orders are the same for $P^1$ and $P^2$ polynomial approximations. When $\xi=2/3$, the convergence order gets worse. From the above, we can conclude that a reduced convergence order can be obtained due to the reduced regularity.

\begin{figure}[H]
\centering
\scalebox{0.3}{
\includegraphics[width=20cm]{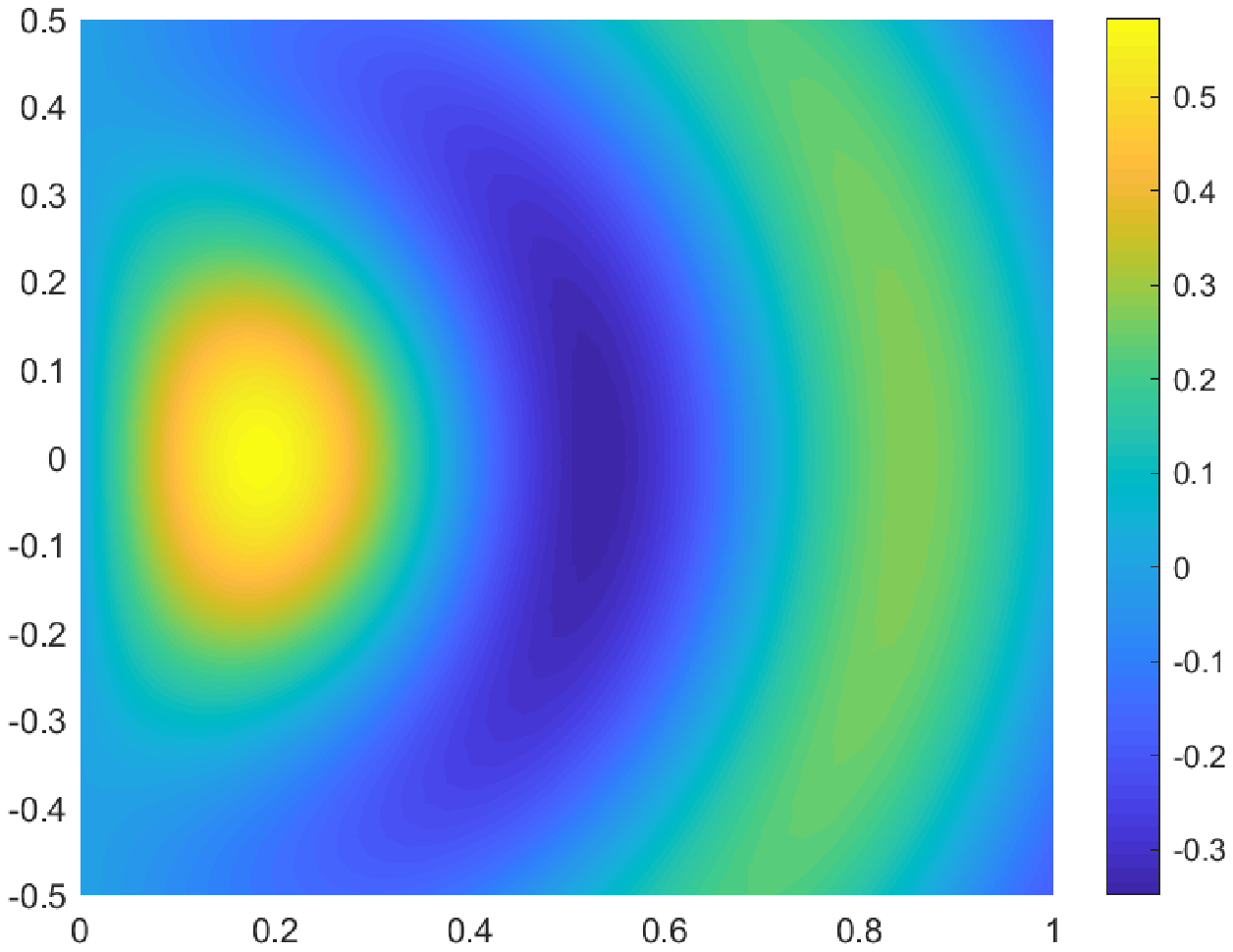}
}
\scalebox{0.3}{
\includegraphics[width=20cm]{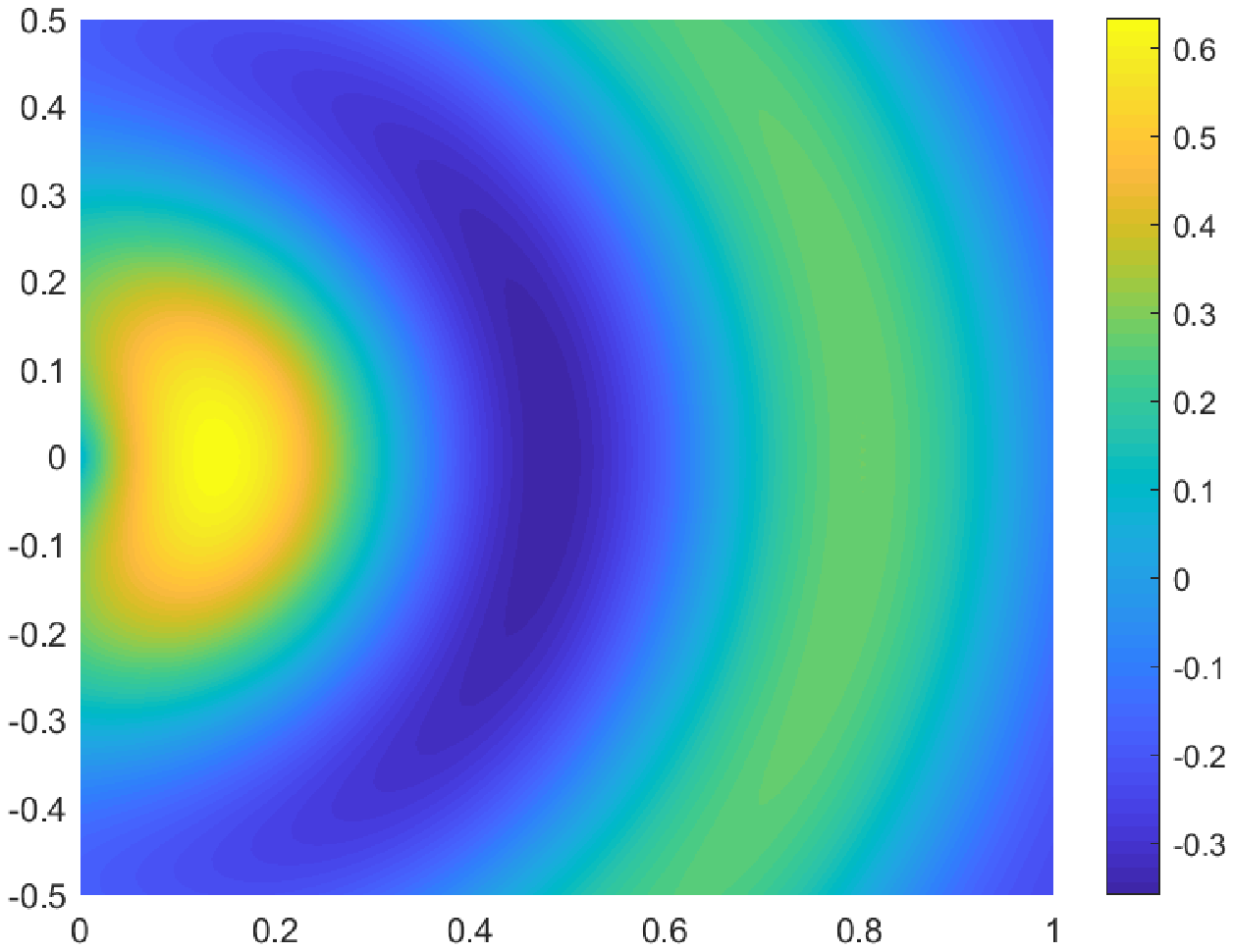}
}
\caption{Numerical solution for $\xi=1,2/3$.}
\label{ex2-solution1}
\end{figure}

\begin{figure}[H]
\centering
\includegraphics[width=7cm]{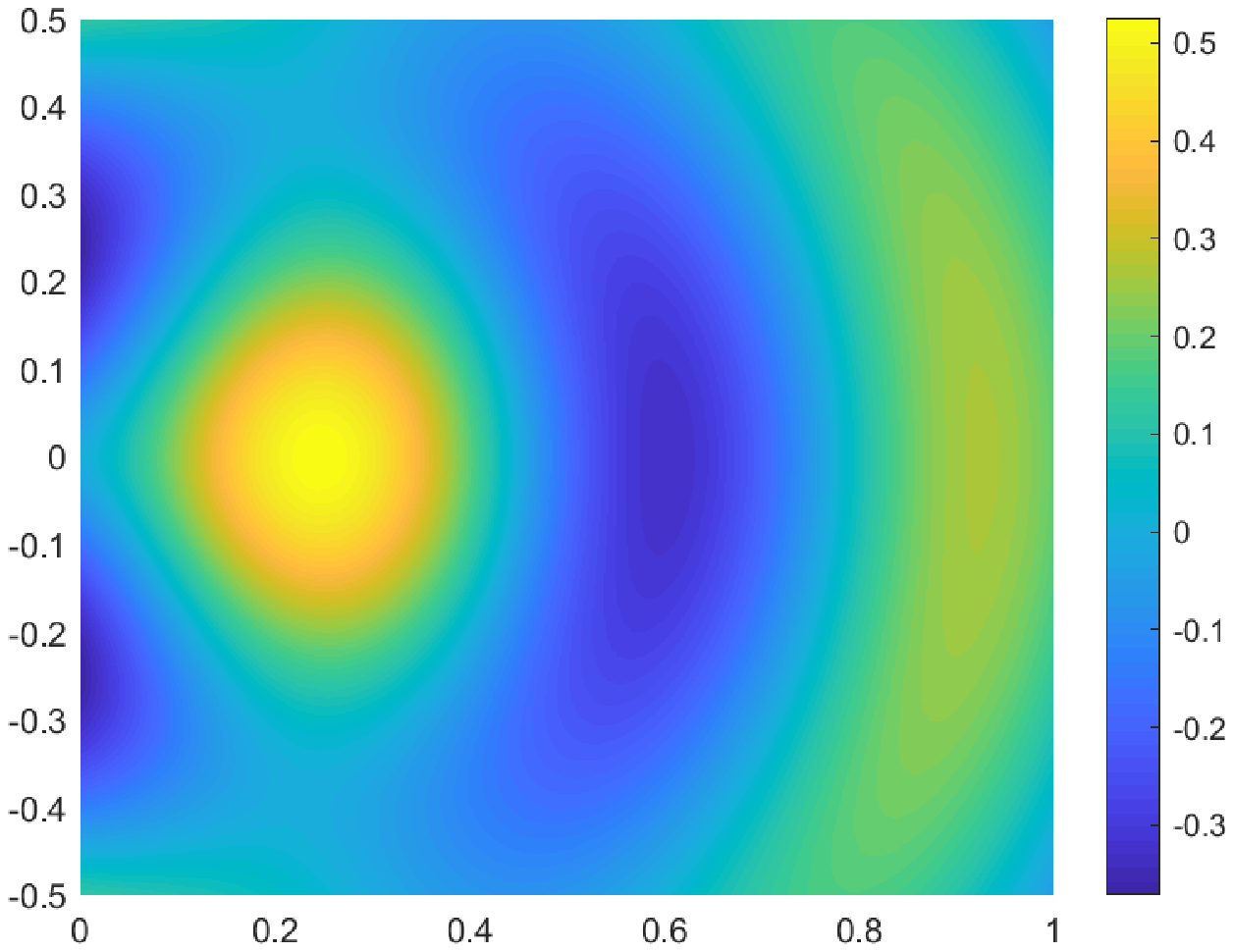}
\caption{Numerical solution for $\xi=3/2$.}
\label{ex2-solution2}
\end{figure}

%\begin{figure}[H]
%    \centering
%    \begin{minipage}[b]{0.4\textwidth}
%      \includegraphics[width=1\textwidth]{ex2_P1xi1}
%    \end{minipage}%
%    \begin{minipage}[b]{0.4\textwidth}
%      \includegraphics[width=1\textwidth]{ex2_P2xi1}
%    \end{minipage}
%    \begin{minipage}[b]{0.4\textwidth}
%      \includegraphics[width=1\textwidth]{ex2_P1xi23}
%    \end{minipage}
%     \begin{minipage}[b]{0.4\textwidth}
%      \includegraphics[width=1\textwidth]{ex2_P2xi23}
%    \end{minipage}
%     \begin{minipage}[b]{0.4\textwidth}
%      \includegraphics[width=1\textwidth]{ex2_P1xi32}
%    \end{minipage}
%    \begin{minipage}[b]{0.4\textwidth}
%      \includegraphics[width=1\textwidth]{ex2_P2xi32}
%    \end{minipage}%
%  \caption{Errors of $\|u-u_h\|_0$,$\|\bm{p}-\bm{p}_h\|_0$ for $\kappa=50$ by $\mbox{SDG}-P_{1}$,$\mbox{SDG}-P_{2}$ and $\mbox{SDG}-P_{3}$ approximations (left, top to bottom). Errors of $\|u-u_h\|_0$,$\|\bm{p}-\bm{p}_h\|_0$ for $\kappa=100$ by $\mbox{SDG}-P_{1}$,$\mbox{SDG}-P_{2}$ and $\mbox{SDG}-P_{3}$ approximations (right, top to bottom).}
%\end{figure}

\begin{table}[H]
\centering
\begin{tabular}{c|cccc|cccc}
\hline
\multicolumn{5}{c|}{$P^1$} & \multicolumn{4}{c}{$P^2$}\\
\hline
$h$& $\|u-u_h\|_0$ & order & $\|\bm{p}-\bm{p}_h\|_0$ & order & $\|u-u_h\|_0$ & order & $\|\bm{p}-\bm{p}_h\|_0$ & order \\
\hline
$0.5000$  & 1.72E-1  &  --    & 1.65E-1   & --     & 4.19E-2  &  --     & 4.78E-2   &--     \\
$0.2500$  & 3.91E-2  & 2.1358 & 4.43E-2   &1.9018  & 5.60E-3  & 2.9069  & 5.30E-3   &3.1641 \\
$0.1250$  & 1.07E-2  & 1.8736 & 1.18E-2   &1.9047  & 7.46E-4  & 2.9046  & 6.17E-4   &3.1105  \\
$0.0625$  & 2.70E-3  & 1.9730 & 3.00E-3   &1.9702  & 9.49E-5  & 2.9743  & 7.42E-5   &3.0555  \\
$0.0313$  & 6.82E-4  & 1.9933 & 7.58E-4   &1.9920  & 1.19E-5  &2.9937   & 9.13E-6   & 3.0242 \\
$0.0156$  & 1.70E-4  & 1.9983 & 1.90E-4   &1.9979  & 1.49E-6  & 2.9984  & 1.13E-6   &3.0107  \\
\hline
\end{tabular}
\caption{Convergence for $\xi=1$ by using $P^1$ and $P^2$ approximations.}
\label{table:1}
\end{table}

\begin{table}[H]
\centering
\begin{tabular}{c|cccc|cccc}
\hline
\multicolumn{5}{c|}{$P^1$} & \multicolumn{4}{c}{$P^2$}\\
\hline
$h$& $\|u-u_h\|_0$ & order & $\|\bm{p}-\bm{p}_h\|_0$ & order & $\|u-u_h\|_0$ & order & $\|\bm{p}-\bm{p}_h\|_0$ & order \\
\hline
$0.5000$ & 1355E-1  & --      & 2.23E-1 &  --   & 4.78E-2 & --      & 4.87E-2  & --     \\
$0.2500$ & 3.44E-2  & 1.9794  & 5.18E-2 &2.1085 & 6.10E-3 & 2.9783  & 7.30E-3  &2.7415  \\
$0.1250$ & 1.04E-2  & 1.7200  & 1.59E-2 &1.7029 & 8.86E-4 & 2.7756  & 1.10E-3  &2.6751  \\
$0.0625$ & 3.00E-3  & 1.7801  & 4.70E-3 &1.7572 & 1.75E-4 & 2.3379  & 3.44E-4  &1.7298  \\
$0.0313$ & 9.52E-4  & 1.6723  & 1.50E-3 &1.6806 & 5.05E-5 & 1.7936  & 1.23E-4  &1.4858  \\
$0.0156$ & 3.22E-4  & 1.5649  & 4.89E-4 &1.5839 & 1.70E-5 & 1.5683  & 4.44E-5  &1.4685  \\
\hline
\end{tabular}
\caption{Convergence for $\xi=3/2$ by using $P^1$ and $P^2$ approximations.}
\label{table:2}
\end{table}

\begin{table}[H]
\centering
\begin{tabular}{c|cccc|cccc}
\hline
\multicolumn{5}{c|}{$P^1$} & \multicolumn{4}{c}{$P^2$}\\
\hline
$h$& $\|u-u_h\|_0$ & order & $\|\bm{p}-\bm{p}_h\|_0$ & order & $\|u-u_h\|_0$ & order & $\|\bm{p}-\bm{p}_h\|_0$ & order \\
\hline
$0.5000$ & 2.33E-1  &   --   & 1.92E-1   &  --    & 4.68E-2  &  --     & 5.79E-2   & --     \\
$0.2500$ & 7.35E-2  & 1.6675 & 6.88E-2   &1.4832  & 8.70E-3  & 2.4330  & 1.15E-2   &2.3287  \\
$0.1250$ & 3.03E-2  & 1.2799 & 2.71E-2   &1.3428  & 2.30E-3  & 1.8885  & 5.90E-3   &0.9690   \\
$0.0625$ & 1.59E-2  & 0.9317 & 1.64E-2   &0.7273  & 7.36E-4  & 1.6665  & 3.60E-3   &0.7182   \\
$0.0313$ & 9.40E-3  & 0.7498 & 1.13E-2   & 0.5355 & 2.52E-4  & 1.5477  & 2.20E-3   &0.6817   \\
$0.0156$ & 5.80E-3  & 0.6904 & 7.80E-3   &0.5370  & 1.01E-4  & 1.3216  & 1.40E-3   &0.6706   \\
\hline
\end{tabular}
\caption{Convergence for $\xi=2/3$ by using $P^1$ and $P^2$ approximations.}
\label{table:3}
\end{table}

%\multicolumn{2}{|c|}{ \multirow{2}*{$S_i$} }& \multicolumn{4}{c|}{事件} &\multirow{2}*{max}\\

%\begin{figure}
%\centering
%\scalebox{0.3}{
%\includegraphics[width=20cm]{solution50-1.eps}
%}
%\scalebox{0.3}{
%\includegraphics[width=20cm]{solution100-1.eps}
%}
%\caption{Exact and numerical solution for $\kappa=50$ ($k=1$).}
%\label{h100-1}
%\end{figure}
%
%
%\begin{figure}
%\centering
%\scalebox{0.3}{
%\includegraphics[width=20cm]{solution50-2.eps}
%}
%\scalebox{0.3}{
%\includegraphics[width=20cm]{solution100-2.eps}
%}
%\caption{Exact and numerical solution for $\kappa=50$ ($k=1$).}
%\label{h100-1}
%\end{figure}
%
%\begin{figure}
%\centering
%\scalebox{0.3}{
%\includegraphics[width=20cm]{solution50-3.eps}
%}
%\scalebox{0.3}{
%\includegraphics[width=20cm]{solution100-3.eps}
%}
%\caption{Exact and numerical solution for $\kappa=50$ ($k=1$).}
%\label{h100-1}
%\end{figure}

% h-perturbation

%polygonal

%\section{Conclusion}
%
%In this paper we present staggered DG method for Helmholtz equation with large wave number. The method can be flexibly applied to rough grids and hanging nodes can be simply treated as additional vertices. Theoretical analysis shows that the stability and convergence estimates only require that $\kappa h$ be sufficiently small. Our numerical experiments indicate that the

\section*{Acknowledgements}

The research of Eric Chung is partially supported by the Hong Kong RGC General Research Fund (Project numbers 14304217 and 14302018) and CUHK Faculty of Science Direct Grant 2018-19.


\begin{thebibliography}{1}

\bibitem{Ainsworth06} {\sc M. Ainsworth, P. Monk, and W. Muniz},
{\em Dispersive and dissipative properties of discontinuous Galerkin finite element methods for the second-order wave equation}, J. Sci. Comput., 27 (2006), pp. 5--40.

\bibitem{Aziz88} {\sc A. K. Aziz, R. B. Kellogg, and A. B. Stephens},
{\em A two point boundary value problem with a rapidly oscillating solution}, Numer. Math., 53 (1988), pp. 107--121.

\bibitem{BabuskaMelenk97} {\sc I. Babu\u{s}ka and J. M. Melenk},
{\em The partition of unity method}, Internat J. Numer. Methods Engrg., 40 (1997), pp. 727--758.


\bibitem{BabuskaSauter} {\sc I. Babu\u{s}ka and S. A. Sauter},
{\em Is the pollution effect of the FEM avoidable for the Helmholtz equation considering high wave number}, SIAM review, 42 (2000), pp. 451--484.


\bibitem{BabuskaGFEM98} {\sc I. Babu\u{s}ka, U. Banerjee, and J. Osborn},
{\em Generalized finite element method-main ideas, results and perspective}, Int. J. Comput. Methods., 1 (2004), pp. 67--103.

\bibitem{BabuskaGFEM95} {\sc I. Babu\u{s}ka, F. Ihlenburg, E. T. Paik, and S. A. Sauter},
{\em A generalized finite element method for solving the Helmholtz equation in two dimensions with minimal pollution}, Comput. Methods Appl. Mech. Engrg., 128 (1995), pp. 325--359.

\bibitem{Beir13} {\sc L. Beir\~{a}o da Veiga, F. Brezzi, A. Cangiani, G. Manzini, L. D. Marini, and A. Russo},
{\em Basic principles of virtual element method}, Math. Models Methods Appl. Sci., 23 (2013), pp. 199--214.

\bibitem{Cangiani16} {\sc A. Cangiani, E. H. Georgoulis, T. Pryer, and O. J. Sutton},
{\em A posteriori error estimates for the virtual element method}, Numer. Math., 137 (2017), pp. 857--893.

\bibitem{Chang90} {\sc C. L. Chang},
{\em A least-squares finite element method for the Helmholtz equation}, Comput. Methods Appl. Mech. Engrg., 83 (1990), pp. 1--7.


\bibitem{ChenLuXu13} {\sc H. Chen, P. Lu, and X. Xu},
{\em A hybridizable discontinuous Galerkin method for the Helmholtz equation with high wave number}, SIAM. J. Numer. Anal., 51 (2013), pp. 2166--2188.

\bibitem{ChenQiu16} {\sc H. Chen and W. Qiu},
{\em A first order system least squares method for the Helmholtz equation}, J. Comput. Appl. Math., 309 (2017), pp. 145--162.

\bibitem{ChungCockburn14} {\sc E. T. Chung, C. Cockburn, and G. Fu},
{\em The staggered DG method is the limit of a hybridizable DG method}, SIAM
J. Numer. Anal., 52 (2014), pp. 915--932.


\bibitem{ChungEngquistwave} {\sc E. T. Chung and B. Engquist},
{\em Optimal discontinuous Galerkin methods for wave propagation}, SIAM
J. Numer. Anal., 44 (2006), pp. 2131--2158.

\bibitem{ChungEngquist} {\sc E. T. Chung and B. Engquist},
{\em Optimal discontinuous Galerkin methods for the acoustic wave equation in higher dimensions}, SIAM. J. Numer. Anal., 47 (2009), pp. 3820--3848.

\bibitem{ChungCiarlet13Yu} {\sc E. T. Chung, P. Ciarlet, Jr., and T. Yu},
{\em Convergence and superconvergence of staggered discontinuous Galerkin methods for the three-dimensional Maxwell's equations on Cartesian grids}, J. Comput. Phys., 235 (2013), pp. 14--31.

\bibitem{ChungKimWidlund13} {\sc E. T. Chung, H. Kim, and O. B. Widlund},
{\em Two-level overlapping schwarz algorithms for a staggered discontinuous Galerkin method}, SIAM J. Numer. Anal., 51 (2013), pp. 47--67.

\bibitem{ChungParkLina} {\sc E. T. Chung, E.-J. Park, and L. Zhao},
{\em Guaranteed a posteriori error estimates for a staggered discontinuous Galerkin method}, J. Sci. Comput., 75 (2018), pp. 1079--1101.

\bibitem{ChungQiu} {\sc E. T. Chung and W. Qiu},
{\em Analysis of an SDG method for the incompressible Navier-Stokes equations}, SIAM. J. Numer. Anal., 55 (2017), pp. 543--569.

\bibitem{Ciarlet78} {\sc G. Ciarlet},
{\em The finite element methods for elliptic problems}, North-Holland Publishing, Amsterdam, 1978.


%\bibitem{CummingsFeng06} {\sc P. Cummings and X. Feng},
%{\em Sharp regularity coefficients estimates for complex-valued acoustic and elastic Helmholtz equations}, Math. Models Methods Appl. Sci., 16 (2006), pp. 139--160.

\bibitem{Demkowicz12} {\sc L. Demkowicz, J. Gopalakrishnan, I. Muga, and J. Zitelli},
{\em Wavenumber explicit analysis of a DPG method for the multidimensional Helmholtz equation}, Comput. Methods Appl. Mech. Engrg., 213-216 (2012), pp. 126--138.


%\bibitem{DouglasSheen01} {\sc J. Douglas, J. E. Santos, and D. Sheen},
%{\em Nonconforming Galerkin methods for the Helmholtz equation}, Numer. Methods Partial Differential Equations, 17 (2001), pp. 475--494.

\bibitem{Douglas93} {\sc J. Douglas, J. E. Santos, D. Sheen, and L. Schreiyer},
{\em Frequency domain treatment of one-dimensional scalar waves}, Math. Models Methods in Appl. Sci., 3 (1993), pp. 171--194.


\bibitem{FengWu09} {\sc X. Feng and H. Wu},
{\em Discontinuous Galerkin methods for the Helmholtz equation with large wave number}, SIAM J. Numer. Anal., 47 (2009), pp. 2872--2896.

\bibitem{FengWu11} {\sc X. Feng and H. Wu},
{\em hp-discontinuous Galerkin methods for the Helmholtz equation with large wave number}, Math. Comp., 80 (2011), pp. 1997--2024.


\bibitem{FengXing13} {\sc X. Feng and Y. Xing},
{\em Absolutely stable local discontinuous Galerkin methods for the Helmholtz equation with large wave number}, Math. Comp., 82 (2013), pp. 1269--1296.

\bibitem{Gallistl15} {\sc D. Gallistl and D. Peterseim},
{\em Stable multiscale Petrov-Galerkin finite element method for high frequency acoustic scattering}, Comput. Methods Appl. Mech. Engrg., 295 (2015), pp. 1--17.

\bibitem{Griesmaier11} {\sc R. Griesmaier and P. Monk},
{\em Error analysis for a hybridizable discontinuous Galerkin method for the Helmholtz equation}, J. Sci. Comput., 49 (2011), pp. 291--310.

\bibitem{Harari91} {\sc I. Harari and T. R. Hughes},
{\em Finite element methods for the Helmholtz equation in an exterior domain: model problems}, Comput. Methods Appl. Mech. Engrg., 87 (1991), pp. 59--96.

\bibitem{Harari92} {\sc I. Harari and T. R. Hughes},
{\em Analysis of continuous formulatons underlying the computation of time-harmonic acoustics in exterior domains}, Comput. Methods Appl. Mech. Engrg., 97 (1992), pp. 103--124.


\bibitem{Hiptmair11} {\sc R. Hiptmair A. Moiola, and I. Perugia},
{\em Plane wave discontinuous Galerkin methods for the 2D Helmholtz equation: analysis of the p-version}, SIAM J. Numer. Anal., 49 (2011), pp. 264--284.

\bibitem{KimChungLam16} {\sc H. Kim, E. T. Chung, and C. Y. Lam},
{\em Mortar formulation for a class of staggered discontinuous Galerkin methods}, Comput. Math. Appl., 71 (2016), pp. 1568--1585.

\bibitem{KimChungLee} {\sc H. Kim, E. T. Chung, and C. S. Lee},
{\em A staggered discontinuous Galerkin method for the Stokes system}, SIAM J. Numer. Anal., 51 (2013), pp. 3327--3350.

\bibitem{LeeKim16} {\sc J. J. Lee and H. Kim},
{\em Analysis of a staggered discontinuous Galerkin method for linear elasticity}, J. Sci. Comput., 66 (2016), pp. 625--649.

\bibitem{MelenkBabuska96} {\sc J. M. Melenk and I. Babu\u{s}ka},
{\em The partition of unity finite element method: Basic theory and applications}, Comput. Methods Appl. Mech. Engrg., 139 (1996), pp. 289--314.

\bibitem{MelenkSauter13} {\sc J. M. Melenk, A. Parsania, and S. Sauter},
{\em General DG-methods for highly indefinite Helmholtz problems}, J. Sci. Comput., 57 (2013), pp. 536--581.

\bibitem{MelenkSauter10} {\sc J. M. Melenk and S. Sauter},
{\em Convergence analysis for finite element discretizations of the Helmholtz equation with Dirichlet-to-Neumann boundary conditions}, Math. Comp., 79 (2010), pp. 1871--1914.

\bibitem{MelenkSauter11} {\sc J. M. Melenk and S. Sauter},
{\em Wavenumber explicit convergence analysis for Galerkin discretizations of the Helmholtz equation}, SIAM J. Numer. Anal., 49 (2011), pp. 1210--1243.

\bibitem{MonkWang99} {\sc P. Monk and D.-Q. Wang},
{\em A least-squares method for the Helmholtz equation}, Comput. Methods Appl. Mech. Engrg., 175 (1999), pp. 121--136.


\bibitem{MuWangYe15} {\sc L. Mu, J. Wang, and X. Ye},
{\em A new weak Galerkin finite element method for the Helmholtz equation}, IMA J. numer. Anal., 35 (2015), pp. 1228--1255.


%\bibitem{WuCIPFEM14} {\sc H. Wu},
%{\em Pre-asymptotic error analysis of CIP-FEM and FEM for the Helmholtz equation with high wave number. Part \uppercase\expandafter{\romannumeral1}: Linear version}, IMA J. Numer. Anal., 34 (2014), pp. 1266--1288.

\bibitem{LinaParkStoke-tri} {\sc L. Zhao and E.-J. Park},
{\em Fully computable bounds for a staggered discontiuous Galerkin method for the Stokes equations}, Comput. Math. Appl., 75 (2018), pp. 4115--4134.

\bibitem{LinaPark} {\sc L. Zhao and E.-J. Park},
{\em A staggered discontinuous Galerkin method of
minimal dimension on quadrilateral and polygonal meshes}, SIAM J. Sci. Comput., 40 (2018), pp. 2543--2567.

\bibitem{LinaParkconvection} {\sc L. Zhao and E.-J. Park},
{\em A priori and a posteriori error analysis for a staggered discontinuous Galerkin method for convection dominant diffusion equations}, J. Comput. Appl. Math., 346 (2019), pp. 63--83.

\bibitem{LinaParkShin} {\sc L. Zhao, E.-J. Park, and D.-w. Shin},
{\em A staggered discontinuous Galerkin method for the Stokes equations on general meshes}, Comput. Methods Appl. Mech. Engrg., 345 (2019), pp. 854--875.


\bibitem{ZhuDu15} {\sc L. Zhu and Y. Du},
{\em Pre-asymptotic error analysis of hp-interior penalty discontinuous Galerkin methods for the Helmholtz equation with large wave number}, Comput. Math. Appl., 70 (2015), pp. 917--933.

\bibitem{ZhuWu13} {\sc L. Zhu and H. Wu},
{\em Pre-asymptotic error analysis of CIP-FEM and FEM for Helmholtz equation with high wave number. Part \uppercase\expandafter{\romannumeral2}: hp version}, SIAM J. Numer. Anal., 51 (2013), pp. 1828--1852.



\end{thebibliography}
\end{document}